\documentclass{amsart}
\usepackage[margin=1in]{geometry}
\usepackage{amsmath}
\usepackage{amssymb}
\usepackage{amsthm}
\usepackage{amscd}
\usepackage{enumerate}
\usepackage{xcolor}

\usepackage{graphicx}
\usepackage{amsmath,amsthm,amsfonts,amscd,amssymb,comment,eucal,latexsym,mathrsfs}
\usepackage{stmaryrd}
\usepackage[all]{xy}

\usepackage{epsfig}

\usepackage[all]{xy}
\xyoption{poly}
\usepackage{fancyhdr}
\usepackage{wrapfig}
\usepackage{epsfig}

\usepackage{hyperref}

\theoremstyle{plain}
\newtheorem{thm}{Theorem}[section]
\newtheorem{introthm}{Theorem}

\newtheorem{prop}[thm]{Proposition}
\newtheorem{lem}[thm]{Lemma}
\newtheorem{cor}[thm]{Corollary}

\theoremstyle{definition}
\newtheorem{defn}[thm]{Definition}
\theoremstyle{remark}
\newtheorem{remark}{Remark}

\newtheorem{notation}{Notation}

  \def\C{{\mathbb{C}}}          \def\M{{\mathbb{M}}} \def\N{{\mathbb{N}}}

 \def\cB{{\mathcal{B}}}     \def\cG{{\mathcal{G}}} \def\cH{{\mathcal{H}}}   \def\cK{{\mathcal{K}}} \def\cL{{\mathcal{L}}} \def\cM{{\mathcal{M}}} \def\cN{{\mathcal{N}}}     \def\cS{{\mathcal{S}}} \def\cT{{\mathcal{T}}}

\newcommand{\set}[1]{\left\{ #1\right\}}

\title{Quantum Chromatic Numbers of Products of Quantum Graphs}

\author{Rolando de Santiago}
\address{California State University, Long Beach, Long Beach CA 90840}
\email{rolando.desantiago@csulb.edu}
\urladdr{https://home.csulb.edu/~rolando.desantiago@csulb.edu/}

\author{A. Meenakshi McNamara}
\address{}
\email{m2mcnamara@uwaterloo.ca}
\urladdr{}

\begin{document}

\begin{abstract}
    In this article, we define the Cartesian, Categorical, and Lexicographic, and Strong products of quantum graphs.  We provide bounds on the quantum chromatic number of these products in terms of the quantum chromatic number of the factors. To adequately describe bounds on the lexicographic product of quantum graphs, we provide a notion of a quantum $b$-fold chromatic number for quantum graphs.    
\end{abstract}
\maketitle

\section*{Introduction and Statements of Results}
Graph colorings and chromatic numbers arise naturally as a measure of the complexity of a graph. Recently, quantum versions of chromatic numbers, framed as optimal strategies to win a non-local version of a graph coloring game, have risen to great prominence.  These, along with other non-synchronous non-local games have been the object of intense study in recent years due their significance in fields of mathematics and physics related to quantum information theory.   Non-local games are closely related to questions of Tsirelson's problem, while in operator algebras, the  Connes Embedding Problem and Kirchberg conjecture are known to be logically equivalent to Tsirelson's problem \cite{JungeAndFriendsConnesEmbedding,FrTsirelsonsAndKirchBerg, OzawaCEP}.  The recent result of Ji-Natarajan-Vidick-Wright-Yuen  using the framework of non-local games to show that MIP$^*$=RE  further illustrates the broad impacts of the theory of non-local games\cite{ji2022mipre}. 

Non-local games typically involve a pair of players working in tandem to provide correct answers to questions posed by a neutral referee.  Before play, the players may devise an agreed-upon strategy to correlate their responses; however, they are no longer allowed to communicate once the game begins.  The non-locality of this game, i.e.~the inability of the players to communicate, makes winning multiple rounds of the game challenging in general.  In cases where the players make use of some shared entangled resource, the players can drastically improve performance by correlating their behaviors.  For example, in graph coloring games quantum chromatic numbers of graphs  can be significantly smaller than  classical chromatic numbers of graphs (see \cite{OdditiesOQ, og_qchromatic_graph}).

We are primarily interested in extensions of the quantum graph coloring game defined in \cite{og_qchromatic_graph} to quantum/non-commutative graphs as were developed by \cite{GanesanHarrisQuantumToClassical}. Non-commuative graphs have garnered significant interest in quantum information theory as they arise as the analog of the confusability graph of a quantum channel \cite{noncomm_graph_Duan_2013}.  While several equivalent descriptions of quantum graphs exist in the literature (see for example, \cite{Priyanga} ), in this article we adopt the framework laid out in \cite{weaver_qgraphs_relations} where the aforementioned author considers quantum sets and quantum relations through the paradigm of von Neumann algebras. So far, only quantum graphs on $M_2$ have been fully classified \cite{Matsuda_2022}, so there is much more work to be done in this relatively young field. Further, as classical quantum graphs can be formulated in the language of quantum graphs, these studies further our understanding of the quantum coloring problem on classical graphs. This is of particular interest as \cite{universality_graph_homo_games} showed an equivalence between all synchronous non-local games and quantum coloring games.

Specifically, we concern ourselves with the behavior of the chromatic numbers of quantum graphs. In \cite{graph_products}, the authors define extensions of the Cartesian and categorical products, two of the four classical graph products, to non-commutative graphs (or more specifically submatricial systems) and describe the behavior of the quantum chromatic number for the resulting graph products.  We take their work as inspiration to extend the definition to the categorical and Cartesian products of quantum graphs. Furthermore, we provide a novel definition for the \emph{lexicographic} and \emph{strong} products of quantum graphs, thereby providing extensions of all four fundamental graph products, and proceed with an analysis of the chromatic number. 

As part of the analysis of the lexicographic product, we introduce a quantum version of the $b$-fold chromatic number.  The $b$-fold chromatic number arises naturally in the study of fractional colorings, which have been studied in the literature, see for example \cite{fractional_graphtheory_book}.  Quantum versions of fractional numbers such as tracial and projective ranks \cite{fractional}. In section \ref{section:b-fold}, we give an explicit definition of the $\chi_{b,q}$ for quantum graphs. 
In classical graph theory, it is known that the behavior of the chromatic number of a lexicographic product is intimately connected to the $b$-fold chromatic number. Specifically, for any graphs $G_1$, and $G_2$, $\chi(G_1[G_2])= \chi_{b}(G_1)$ where $b= \chi(G_2)$, see \cite[Theorem 3]{lex_prod_GELLER197587}.  We are able to prove a variation of this result for the lexigraphical product of quantum graphs. 

\begin{introthm}[{Theorem \ref{thm:lexicographicBoundedByProduct}}]\label{thm:Main_Lex_prod}
Let $\mathcal{G}$ and $\mathcal{H}$ be quantum graphs. Then $\chi_q(\mathcal{G}[\mathcal{H}])\leq \chi_{b,q}(\mathcal{G})$ where $b=\chi_q(\mathcal{H})$
\end{introthm}

We strongly suspect that the inequality in Theorem \ref{thm:Main_Lex_prod} should be an equality, thereby emulating the classical setting. However, at the present time
 we are unable to verify this and leave it as an open question.

In section \ref{sec:BoundsForCategoricalAndCartesianProducts}, we turn our attention to the Cartesian and categorical products of quantum graphs.  Variations of   {H}edetniemi's conjecture and Sabidussi's Theorem have been studied in the context of quantum chromatic numbers for graphs, see \cite{SabidussiChromaticServini} and \cite{graph_products}.
We obtain our results by closely emulating the work of Kim and Mehta for a special case of noncommutative graphs, namely submatricial systems, to define Categorical and Cartesian products of quantum graphs
(see Section \ref{sec:Prelim}).  Firstly, we examine Sabidussi's Theorem for the categorical product of quantum graphs.
\begin{introthm}[{Theorem {\ref{thm:BoundCartesianProduct}}}]\label{thm:Main_Cat_Prod}
Let $\mathcal{G}$ and $\mathcal{H}$ be quantum graphs. $$\max\set{\chi_{q}(\mathcal{G}), \chi_{q}(\mathcal{H})} \leq \chi_{q}(\mathcal{G}\Box \mathcal{H})$$
\end{introthm}

We also are able to prove the following version of Hedetniemi's conjecture for quantum graphs.
\begin{introthm}[{Theorem \ref{thm:BoundCategoricalProduct}}]\label{thm:Main_Cartesian_Prod}
Let $\mathcal{G}$ and $\mathcal{H}$ be quantum graphs.
$$\chi_q(\mathcal{G}\times\mathcal{H}) \leq \min\{\chi_q(\mathcal{G}), \chi_q(\mathcal{H})\}$$

\end{introthm}
Lastly, we prove a standard bound for the strong product of quantum graphs. 
\begin{introthm}
    Let $\mathcal{G}$ and $\mathcal{H}$ be quantum graphs.  
    \[\max\{\chi_{q}(\mathcal{G}), \chi_q(\mathcal{H})\}\leq \chi_q(\mathcal{G}\boxtimes\mathcal{H})\leq \chi_q(\mathcal{G})\chi_q(\mathcal{H}) \]
\end{introthm}
We remark that the results above hold for the quantum approximate and quantum commuting chromatic numbers as well.

The organization of our paper is as follows: in Section \ref{sec:Prelim}, we give a brief description of the notion of quantum graphs in the sense of \cite{weaver_qgraphs_relations}, and an equivalent definition quantum graph colorings. Section \ref{sec:GraphProducts} provides definitions of the Cartesian, categorical, lexicographic, and strong graph products. Furthermore, we verify that these indeed do produce quantum graphs in the sense of Section \ref{sec:Prelim}. Section \ref{section:b-fold} is devoted to establishing several fundamental properties for our definition of the local, quantum, quantum approximate, and quantum commuting $b$-fold chromatic number. In particular, we provide various equivalent descriptions $\chi_{q,t}$, and prove that $\chi_{q,t}$ behaves similarly to the classical version. Finally, Section
 \ref{sec:lexicographicBound} gives the proofs of the main results.

\subsection*{Acknowledgements}The authors would like to thank Priyanga Ganesan and Thomas Sinclair  for their patience and thoughtful conversations  with this project.  

\tableofcontents

\section{Preliminaries and Definitions}\label{sec:Prelim}
For our purposes, all classical graphs considered here will be finite, simple, undirected graphs; that is, we will only consider finite vertex set,  and an irreflexive, symmetric relation on the vertices.  $G=(V(G),E(G))$ will denote a graph, with $V(G)$ its vertex set, and $E(G)$ its edge set.  Whenever $(v,w)\in E(G)$, we will write $v\sim w$ in $G$.  When it is clear from context, we will write $G=(V,E)$ in place of $G=(V(G),E(G))$.

$E_{i,j}\in M_{n}(\C)$ will be the elementary matrix with a 1 in the $i,j$-th entry and zero elsewhere.  Unless specified, $H$  will denote a (complex) Hilbert space and $\cB(H)$ denotes the bounded linear operators on ${H}$. Fixing a non-empty subset $\mathcal{S}\subseteq \cB(H)$, $\cS':=\set{Y\in \cB(H): YS=SY \quad \forall \, S\in \mathcal{S} }$ will denote the commutant of $\cS$.  A von Neumann algebra $\mathcal{M}$ is a unital $*$-closed subset with $\cM''=\cM$.  A von Neumann algebra $\cM\subseteq \mathcal{B}(H) $ where ${H}$ is separable is called finite if there exists a normal tracial faithful state $\tau:\cM\to \C$\footnote{More generally, $M$ is called finite if whenever $p\in M$ is a projection such that there exists $v\in M$ so that $v^*v=p$ and $vv^*=1$, it follows that $p=1$.}.  Finally, we say that a finite von Neumann algebra $\mathcal{M}$  embeds into $\mathcal{R}^\omega$ if there exists an injective,  trace preserving $*$-homomorphism from $\mathcal{M}$ into some ultrapower of the hyperfinite II$_1$ factor $\mathcal{R}$. For details on ultrapowers of von Neumann algebras and the the hyperfinite II$_1$ factor, see \cite{OzawaCEP}.

There are a litany of definitions of quantum graphs (see \cite{firstQuantumGraph, noncomm_graph_Duan_2013, weaver_qgraphs_relations, comp_q_func_Musto_2018, bigalois_graphIso_2019},), and a robust analysis of their equivalences can be found in \cite{Priyanga}.   
 We adopt the following description of quantum graphs  given by Weaver (see \cite{weaver_qgraphs_relations} ) as it is amenable to the study of quantum coloring problems. 
An (irreflexive) \textbf{quantum graph} $\mathcal{G}$ is a tuple $(\cS, \mathcal{M}, \cB(H))$ where: $\mathcal{M}$ is a von Neumann subalgebra of $\cB(H)$; $\mathcal{S}\subseteq \cB(H)$ is an operator space that is closed under adjoint, is a bimodule over $\mathcal{M}'$, and is contained in $(\mathcal{M}')^\perp$ (relative to some trace on $\cB(H)$).
Note that an irreflexive quantum graph on $M_n(\C^n)$ is a self-adjoint, traceless operator subspace of $M_n(\C)$, a definition of a non-commutative graph frequently encountered in the literature (e.g. \cite{Stahlke_2016_noncomm_graphs}). In \cite{stahlke}, Stahlke introduced a graph coloring game that takes advantage of quantum entanglement.  Brannan, Ganesan, and Harris demonstrate that the quantum chromatic number of a graph can be determined entirely in terms of projection valued measures, which allows for the extension of quantum colorings to quantum graphs in the sense above. 

Given a quantum graph $\mathcal{G} = (\mathcal{S}, \mathcal{M}, \cB(H))$ and $t\in \{loc, q, qa, qc\}$, there is a \textbf{$t$ c-coloring} of $\mathcal{G}$ if there exists a finite von Neumann algebra $\mathcal{N}$ and $\{P_a\}_{a=1}^c\subseteq \mathcal{M}\bar\otimes\mathcal{N}$ such that
          \begin{enumerate}
        \item $\displaystyle P_a^2 = P_a = P_a^*$ for $1\leq a\leq c$
        \item $\displaystyle \sum_{i=1}^cP_a = I_{\mathcal{M}\bar{\otimes}\mathcal{N}}$
        
        \item $\displaystyle P_a(X\otimes I_{\mathcal{N}})P_a = 0, \forall X\in \mathcal{S}$ and $1\leq a\leq c$,
\end{enumerate}
and  where  $\dim \mathcal{N} = 1$ if $t = loc$,  $\dim\mathcal{N}<\infty$ if $t=q$,  $\mathcal{N}$ embeds into $\mathcal{R}^\omega$ if $t=qa$ if, or $\mathcal{N}$ is finite if $t = qc$. $\chi_t(\mathcal{G}) $ is defined as  the minimum $c\in \N$ such that there exists a $t$-$c$-coloring  of $\mathcal{G}$, when one exists, or infinity otherwise. $\chi_{loc}(\mathcal{G})$, $\chi_{q}(\mathcal{G})$, $\chi_{qa}(\mathcal{G})$, $\chi_{qc}(\mathcal{G})$ are respectively called the \textbf{local}, \textbf{quantum}, \textbf{quantum approximate}, \textbf{quantum commuting} \textbf{chromatic numbers of $\mathcal{G}$}, respectively. When $\mathcal{G}$ arises from a classical graph,   these can be phrased in terms of correlation sets, see \cite{fractional}
and \cite{GanesanHarrisQuantumToClassical}. In this formulation, we have that the quantum approximate colorings come from the closure of the correlation sets for the quantum strategies.


As an example, a classical graph  $G=(V,E)$ can be expressed as a quantum graph. Specifically, if  $G$ is a graph and   $V=\set{v_1,\ldots, v_n}$, letting 
$\mathcal{S}_\cG :=\operatorname{span}\set{E_{i,j}\in M_{n} : v_i\sim v_j \text{ in }G  }$,  $D_n\subseteq \cB(\C^n)=M_{n}(\C)  $ be the set diagonal matrices, and
 $\mathcal{G}= (\mathcal{S}_\cG , D_n, \cB(\C^n))$, one can check that $\mathcal{G}$ is a quantum graph as described above.  It is well-known that if $G$ and $H$ are graphs, then an isomorphism of operators spaces $\cS_\cG $ and $\cS_\cH$ corresponds to an isomorphism of graphs  \cite{isomorphism_classical}.  It is also the case that any quantum graph $\mathcal{G}= (\mathcal{S}, D_n, M_n(\C)))$ has the property that $\mathcal{S}=\mathcal{S}_\cG $ for some classical graph $G$. Hence, all ``abelian'' quantum graphs are classical graphs.  Thus, a ``purely quantum'' graph $\mathcal{G}= (\mathcal{S},\mathcal{M},\cB(H))$ arises when $\mathcal{M}$ is genuinely non-commutative.
 We call $(\cS_\cG , D_n,M_{n}(\C)$ the \textbf{quantum graph associated to a classical graph}. When $\mathcal{G}$ arises from a classical graph $G$ in the sense above,  $\chi_{loc}(\mathcal{G})= \chi(G)$ and in general $\chi_t(\cG) = \chi_t(G)$ where $t$ determines the type of nonlocal game.  Moreover, when $\mathcal{G}=(\mathcal{S}, M, \cB(H))$ is such that ${H}$ is finite dimensional,  $\chi_t(\mathcal{G})<\infty$ for all $t\in \set{q,qa,qc}$. Furthermore, if $\cG$ is a complete quantum graph then $\chi_{loc}(\mathcal{G})$ will be finite if and only if $M $ is abelian; i.e.~$\mathcal{G}$ corresponds to a classical graph. 

The notion of quantum graph homomorphisms was introduced around the same time as quantum graphs, for instance in \cite{weaver_qgraphs_relations,Stahlke_2016_noncomm_graphs}. We are interested in $t$-quantum graph homomorphisms, i.e. non-local game versions of a quantum graph homomorphism. In \cite{quantum_graph_homo_op_sys} this was explored for classical graphs in their quantum graph representation. In \cite{GanesanHarrisQuantumToClassical}, non-local games for homomorphisms between quantum and classical graphs were defined by determining a way to include quantum inputs to the game. These definitions can combine to give a non-local game version of quantum graph homomorphisms as follows:
    There  \textbf{quantum graph homomorphism of type} $t\in \{loc, q, qa, qc\}$ from $\mathcal{G}$ to $\mathcal{H}$ (denoted by $\mathcal{G}\to_t\mathcal{H}$) if there is a finite von Neumann algebra $\cN$ corresponding to $t$ and a completely positive trace preserving (CPTP) map $\Phi: \mathcal{M}_\cG \otimes \mathcal{N}_* \to \cM_\cH$ of the form $\Phi(\cdot) = \sum_{i=1}^m F_i(\cdot)F_i^*$ such that 
    $F_i(\cS_\cG \otimes I_\cN)F_j^*\subset \cS_\cH$ and  $F_i(\cM_\cG '\otimes I_\cN)F_j^*\subset \cM_\cH'$ for all $i,j$.

\section{Graphs Products}\label{sec:GraphProducts}

Graph products are binary operations on graphs $G=(V(G),E(G))$ and $H=(V(H),E(H))$ such that the graph product has vertex set $V(G)\times V(H) $; there are however, a number of ways one may reasonably define the edge relations.  For an in-depth discussion on graph products, we direct the interested reader to \cite{HandbookGraphProducts}.   Our approach is inspired by Kim and Mehta's definition of the Cartesian and categorical graph products for submatricial systems \cite{graph_products}.
 We will also define the lexicographic and strong products of quantum graphs, which to our knowledge have not been investigated.

\subsection{Cartesian Graph Product}
 
For classical graphs $G$ and $H$, the \textit{Cartesian product} is the graph $G\square H$ with vertex set $V(G)\times V(H)$ and edge relation given by $(v,a)\sim (w,b)$ if and only if either $v\sim w$ in $G$ and $a=b$, or $v=w$ and  $a\sim b$ in $H$. 

\begin{defn}
[{c.f. \cite[Definition 37]{graph_products}}]\label{cartesion_prod_def}
Suppose $\mathcal{G} = (\cS_\cG,\cM_\cG, \cB(H_\cG))$ and $\mathcal{H} = (\cS_\cG, \cM_\cH, \cB(H_\cH))$ are quantum graphs. The \textbf{Cartesian product} of $\mathcal{G}$ and $\mathcal{H}$ is the quantum graph defined by 
$$\mathcal{K} = \mathcal{G}\square\mathcal{H} = (\cS_\cK, \cM_\cK, \cB(H_\cK))$$
where $\cM_\cK = \cM_\cG \bar{\otimes}\cM_\cH$, $H_\cK = H_\cG\bar{\otimes} H_\cH$ and 
$$\cS_\cK = \cS_\cG\otimes \cM_\cH'+ \cM_\cG'\otimes \cS_\cH.$$
\end{defn}


We now verify that the above graph product does indeed produce a quantum graph.
\begin{prop}\label{cartesian_get_graph_proof}For any irreflexive quantum graphs,  $\mathcal{G}$ and $\mathcal{H}$, the Cartesian product  $\mathcal{G}\square \mathcal{H}  $ is an irreflexive quantum graph.
\end{prop}
\begin{proof}
Let us denote $\mathcal{K}=\mathcal{G}\square \mathcal{H} = (\cS_\cK,
\cM_\cK, \cB(H_\cK))$ where $\cS_\cK, \cM_\cK, $ and $\cB(H_cK)$ are as in Definition \ref{cartesion_prod_def}. 
Note:
\begin{enumerate}
    \item 
    Since $\cS_\cK$ is the sum of the tensor product of operator spaces it is an operator space, and it is clear that we have $\cS_\cK \subseteq \cB(H_\cG)\otimes \cB(H_\cH) = \cB(H_\cG\otimes H_\cH) \subseteq \cB(H_\cK)$ because $\cS_\cG, \cM_\cG' \subseteq \cB(H_\cG)$ and $\cS_\cH, \cM_\cH'\subseteq \cB(H_\cH)$. 
    Further, since $\mathcal{G}$ and $\mathcal{H}$ are quantum graphs, we have $\cS_\cG$ and $\cS_\cH$ are closed under adjoint. Additionally, $\cM_\cG'$ and $\cM_\cH'$ are von Neumannn algebras and thus closed under adjoint. Thus, $\cS_\cK$ is the tensor product of spaces closed under adjoint and is closed under adjoint.
    
    \item $\cM_\cK = \cM_\cG\bar{\otimes}\cM_\cH$ is the von Neumann algebra tensor product of two von Neumann algebras and thus a von Neumann algebra by construction. 
    
    
    
    \item It is known that for von Neumann algebras $\cM_\cG$ and $\cM_\cH$ we have $(\cM_\cG\bar{\otimes}\cM_\cH)' = \cM_\cG'\bar{\otimes}\cM_\cH'$ (see for example \cite{TensorProductvNA}).
    Suppose $A\otimes B,C\otimes D \in \cM_\cG' \otimes \cM_\cH' \subset  \cM_\cK'$. Then we have
    \begin{align*}
        (A\otimes B)\cS_\cK(C\otimes D) &= (A\otimes B) ((\cS_\cG \otimes \cM_\cH' + (\cM_\cG '\otimes \cS_\cH) )(C\otimes D)\\
        &= A\cS_\cG C \otimes BM_\cH'D + A\cM_\cG 'C\otimes B\cS_\cH D.
    \end{align*}
    Since $A,C\in \cM_\cG'$ and $\mathcal{G}$ is a quantum graph, the bicommutant property of $\cS_\cG$ gives us $A \cS_\cG C \subseteq \cS_\cG$, and similar reasoning shows $B\cS_\cH D\subseteq \cS_\cH$.  Finally, it follows that $A\cM_\cG'C\subset \cM_\cG'$ since $\cM_\cG'$ is an algebra. Similarly,  $B\cM_\cH 'D\subset \cM_\cH'$. Thus, $(\cM_\cG'\otimes \cM_\cH') 
    \cS_\cK\mathcal(\cM_\cG '\otimes \cM_\cH') \subset \cS_\cK,$. Since $\cM_\cG'\otimes \cM_\cH'$ is a generating subset of $\cM_\cK'$, we have $\cM_\cK'\cS_\cK\cM_\cK'\subset \cS_\cK$ as well. 

    \item Let $A\otimes B \in \cM_\cG'\otimes \cM_\cH'\subset \cM_\cK'.$ and  $S = G\otimes J + D\otimes H \in \cS_\cK$. Thus 
    \begin{align*}
        tr\big((S)(A\otimes B)^*\big) &= tr\big((G\otimes J+D\otimes H)(A^* \otimes B^*)\big)\\
        &= tr\big(GA^*\otimes JB^*\big) + tr(DA^* \otimes HB^*\big)\\
        &= tr(GA^*)tr(JB^*) + tr(DA^*)tr(HB^*) = 0
    \end{align*}
   since $G\perp A$ and $H\perp B$ by definition. Thus $tr(GA^*) = 0$ and $tr(HB^*) = 0$. Therefore, we have $S$ and $A\otimes B$ are orthogonal for all $A\otimes B\in \cM_\cG'\otimes \cM_\cH'$. Since these  generate $\cM_\cK'$ and $S\in \cS_\cK$, it follows that $\cS_\cK \subseteq(\cM_\cK')^\perp.$
\end{enumerate}
Thus,  $\mathcal{K}$ is a quantum graph.
\end{proof}

\subsection{Categorical Product}

 For classical graphs $G$ and $H$, the \textit{categorical product} is the graph $G\times H$ with vertex set $V(G)\times V(H)$ and edge relation given by $(v,a)\sim (w,b)$ if and only if $v\sim w$ in $G$ and  $a\sim b$ in $H$.
As with the Cartesian product, we can extend  categorical product of graphs to quantum graphs in an analogous manner. 

\begin{defn}[{c.f. \cite[Proposition 45]{graph_products}}]
\label{categor_prod_def}
Suppose $\mathcal{G} = (\cS_\cG , \cM_\cG , \cB(H_\cG ))$ and $\mathcal{H} = (\cS_\cH, \cM_\cH, \cB(H_\cH))$ are quantum graphs. Then we define the \textbf{categorical product} of $\mathcal{G}$ and $\mathcal{H}$ to be 
$$\mathcal{K} = \mathcal{G}\times \mathcal{H} = (\cS_\cK, \cM_\cK, \cB(H_\cK))$$ where $\cM_\cK = \cM_\cG \bar{\otimes}\cM_\cH$, $H_\cK = H_\cG \bar{\otimes} H_\cH$ and 
$$\cS_\cK = \cS_\cG \otimes \cS_\cH.$$
\end{defn}

\begin{prop}
The product defined in \ref{categor_prod_def} produces an (irreflexive) quantum graph.
\end{prop}

This proof can be conducted analogously to that of Proposition \ref{cartesian_get_graph_proof}, and thus we omit the proof.

\begin{remark}
 Definitions \ref{cartesion_prod_def}  and \ref{categor_prod_def} reduces to Kim and Mehta's definition of Cartesian (resp. categorical) product for submatricial systems \cite{graph_products} by taking the space of diagonal matrices in place of $\mathcal{M}'$. We note that the flip automorphism on tensor products identifies  $\mathcal{G}\square \mathcal{H}$ and $ \mathcal{H}\square \mathcal{G}$, as well as $\mathcal{G}\times\mathcal{H} $ with $ \mathcal{H}\times \mathcal{G}$.
Moreover, it follows that the process of taking the Cartesian (resp. categorical) product of classical graphs and forming its corresponding quantum graph is essentially equivalent to forming their corresponding quantum graphs and taking their Cartesian (resp. categorical) product as defined above. A more precise statement can be made in analogy to that in Proposition \ref{prop:checkingClassicalProductLex}.
\end{remark}

\subsection{Lexicographic Product}\label{subsec:Lex}
In addition to generalizing the definitions of Kim and Mehta, we are interested in the lexicographic product of quantum graphs. 
For classical graphs $G$ and $H$, the \textit{lexicographic product} is the graph $G[H]$ with vertex set $V(G)\times V(H)$ and edge relation given by $(v,a)\sim (w,b)$ if and only if either $v\sim w$ in $G$ or $v = w$ and $a\sim b$ in $H$.  This product is unique from the others considered here as it is the only non-commutative product.  One of the features of the lexicographic product is that the clique and independence number of a lexicographic product of graphs based on the inputs is the product of the input graphs \cite{lex_prod_GELLER197587}.  Moreover, the lexicographical product detects ``perfectness''  as the lexicographic product of graphs is perfect precisely when each factor is perfect \cite{perfect_prod_RAVINDRA1977177}.   Notably the problem of determining whether or not a graph arises as a lexicographical product is equivalent to the graph isomophism problem \cite{classical_graph_iso_productsdoi:10.1137/0215045}.

\begin{defn}
\label{lexic_product_def}
The \textbf{lexicographic product} of quantum graphs $\mathcal{G} = \big(\cS_\cG , \cM_\cG , \cB(H_\cG )\big)$ and $\mathcal{H} = \big(\cS_\cH, \cM_\cH, \cB(H_\cH)\big)$ is $$\mathcal{G}[\mathcal{H}] = \mathcal{K} = \big(\cS_\cK, \cM_\cK, \cB(H_\cK)\big)$$ where $\cM_\cK = \cM_\cG \bar{\otimes} \cM_\cH$, $H_\cK = H_\cG \bar{\otimes} H_\cH$ and $$\cS_\cK = 
\cS_\cG \otimes \cB(H_\cH) + \cM_\cH'\otimes 
\cS_\cH.$$
\end{defn}

Verifying that the lexicographic product of irreflexive quantum graphs is again an irreflexive quantum graph is similar to that of all previous products considered here, and thus is left as an exercise to the reader. We now show that the lexicographic product of quantum graphs associated to classical graphs essentially recovers the classical lexicographic product up to the notion of isomorphism introduced in \cite{bigalois_graphIso_2019}.

\begin{prop}
Let $G$ and $H$ be graphs, let $K = G[H]$ be their classical lexicographic product, $\mathcal{G}= (\mathcal{S}_{G}, D_{G}, \mathcal{B}(H_{G}))$, $\mathcal{H}= (\mathcal{S}_{H}, D_{H}, \mathcal{B}(H_{H}))$, $\mathcal{K}= (\mathcal{S}_{K}, D_{K}, \mathcal{B}(H_K))$ be the associated quantum graphs of $G$, $H$, and $K$, respectively.
Additionally, define $\mathcal{K}_q=\mathcal{G}[ \mathcal{H}]= (\mathcal S, \mathcal{M}, B(H_q)) $. Then there exists a unitary $U: H_{K}\to H_q $ such that
$U\cS_K U^*=\cS$, $UD_K U^*=\cM$, and $UB(H_K)U^*=B(H_q)$.
\label{prop:checkingClassicalProductLex}
\end{prop}
\begin{proof}
Note that  $\{\delta_{v,w}\}_{v\in V(G), w\in V(H)}$ is an orthonormal basis for  $H_K$; likewise, $\{\delta_{v} \otimes \delta_{w}\}_{v\in V(G), w\in V(H)}$ is an orthonormal  basis for $H_q$.  We claim that the map $U:H_K\to H_q$ given by $\delta_{v,w}\mapsto \delta_{v} \otimes \delta_{w}$ is a unitary that satisfies the conditions of the theorem.


First, by definition $U$ is a unitary, hence we automatically have $U\cB(H_K)U^* = \cB(H_q)$.
Now, let $\{E_{v_1,v_2} : v_1,v_2\in V(G) \}$, $\{E_{w_1,w_2}: w_1,w_2\in V(H)\}$, and $\{E_{(v_1,w_1),(v_2,w_2)}: v_1, v_2\in V(G), w_1,w_2\in V(H)  \}$,   be  the elementary matrices for $\cB(H_{G})$, $\cB(H_{H})$,  and $\cB(H_K)$, respectively. We note that $D_K = \operatorname{span} (E_{(v,w),(v,w)})$ and that $\cM = \operatorname{span}(E_{v,v}\otimes E_{w,w})$. A quick check shows that we have $UE_{(v,w),(v,w)}U^* = E_{v,v}\otimes E_{w,w}$. Hence $UD_K E^* = \cM$.

  
Finally, $$\cS = \cS_{G}\otimes \cB(H_{H}) + \cM_{G}'\otimes \cS_{H}.$$
Note that elements of $\cS_{G}$, $\cS_{H}$ and $\cS_K$ can have non-zero entries at the $i,j$-th entry if $i\sim j$ in the corresponding classical graph where we identify the vertices with the entries in the matrix.

We can easily see from the expression for $\cS$ that we have $E_{v_1,v_2}\otimes E_{w_1,w_2} \in \cS$ if and only if $w_1\sim w_2\in V(H)$ and $v_1=v_2\in V(G)$ or $v_1\sim v_2\in V(G)$. In particular, this is precisely the definition for the vertices $(v_1,w_1)$ and $(v_2,w_2)$ to be adjacent in $K$ from the definition of the lexicographic product. A quick calculation shows that
$UE_{(v_1,w_1),(v_2,w_2)}U^* = E_{v_1,v_2}\otimes E_{w_1,w_2}$. Hence, $E_{(v_1,w_1),(v_2,w_2)}\in \cS_K$ if and only if $E_{v_1,v_2}\otimes E_{w_1,w_2}\in \cS$, so  
$ U\cS_K U^*=\cS $.
\end{proof}

\subsection{The Strong Product} The strong product is the union of the relations defining the Cartesian and categorical products, and is in some sense the ``most connected'' graph generated by a pair.

\begin{defn}
    Suppose $\cG = (\cS_\cG , \cM_\cG , \cB(H_\cG ))$ and $\cH = (\cS_\cH, \cM_\cH, \cB(H_\cH))$ are quantum graphs. The \textbf{strong product} of $\cG$ and $\cH$ is the quantum graph defined by 
    $$\cK = \cG\boxtimes \cH = (\cS_\cK, \cM_\cK, \cB(H_\cK))$$
    where $\cM_\cK = \cM_\cG \bar{\otimes}\cM_\cH,$ $H_\cK = H_\cG \bar{\otimes}H_\cH$ and 
    $$\cS_\cK = \cS_\cG \otimes \cM'_\cH + \cM_\cG '\otimes \cS_\cH + \cS_\cG \otimes\cS_\cH.$$
    \label{def:strongProduct}
\end{defn}

While this product was not considered explicitly considered by Kim and Mehta, it is a combination of the Cartesian and categorical graph products. Thus, we can combine the results for the aforementioned products to show that for quantum graphs associated to classical graphs the strong product produces the associated quantum graph of the product classical graph defined in the classical setting holds. More explicitly, 

\begin{prop}
    For classical graphs $G$ and $H$ let $K = G\boxtimes H$ be their strong product, and let $\cK = (\cS_K, \cM_K, \cB(H_K))$ be the associated quantum graph of $K$. Additionally, let $\cG$ and $\cH$ be the associated quantum graphs of $G$ and $H$, respectively, and define $ \cK_q =\cG\boxtimes\cH =  (\cS, \cM, \cB(H_q))$. Then 
    there is a unitary $U: H_K \to H_{q}$ implementing an isomorphism between $\cK$ and $\cK_q$.
\end{prop}
\begin{proof}
    This proof can be conducted analogously to the proof of Prop. \ref{prop:checkingClassicalProductLex}.
\end{proof}

\begin{prop}
The product defined in \ref{def:strongProduct} produces an (irreflexive) quantum graph.
\end{prop}

\begin{proof}
This proof can be conducted analogously to the proof of Prop. \ref{cartesian_get_graph_proof}.
\end{proof}

\section{The Quantum \texorpdfstring{$b$}{\textit{b}}-fold Chromatic Number}\label{section:b-fold}

The $b$-fold chromatic number is a modification of the chromatic number that considers colorings of a vertex using $b$ colors. Our main motivation for considering this relatively young invariant of graphs is its natural relation to the chromatic number of lexicographic products of graphs, as will be considered in detail in Section \ref{sec:lexicographicBound}. The $b$-fold chromatic number is also integral to the definition of the fractional chromatic number, a more commonly studied invariant of graphs that is an extension of the traditional chromatic number to rational values, and can provide lower bounds to the chromatic number. We will briefly comment on a non-local version of this graph invariant below as well.

We adopt the following notation in the present section as well as Section \ref{sec:lexicographicBound}.
 \begin{notation}
Let $\set{P_{i}}_{i\in I}$ be a family of projections in  $ \mathcal{B}(H)$. The supremum $\bigvee_{i\in I} P_i$ (respectively, infimum $\bigwedge_{i\in I} P_i$) is the projection onto the closure of $\operatorname{span}\left( \bigcup_{i\in I}P_i H \right)$ in $H$, (respectively, $\operatorname{span} \left( \bigcap_{i\in I}P_i H \right)$ in $H$). In the case where $\set{P_{i}}_{i\in I}$  are projections in a von Neumann algebra $\mathcal{M}\subseteq \mathcal{B}(H)$,   $\bigvee_{i\in I}P_i$ and $ \bigwedge_{i\in I} P_i$ are projections in $\mathcal{M}$, see for example \cite[Chapter V, Proposition 1.1]{TakesakiI}. When $I=\{1,2,\ldots, n\}$ is a finite set, we will write $P_1\vee P_2\vee\cdots \vee P_n$ (respectively, $P_1\wedge P_2\wedge\cdots \wedge P_n$ ) in place of $\bigvee_{i\in I}P_i$ (respectively, $\bigwedge_{i\in I}P_i$). We remark
$\bigvee_{i\in I} P_i$ is the sum of the projections when $\set{P_{i}}_{i\in I}$ is an orthogonal family, 
and  $\bigwedge_{i\in I} P_i$ is the  product of the projections when $\set{P_{i}}_{i\in I}$ is commuting family. 
\end{notation}
\begin{notation}
Let $n\in\N$. Then

\begin{itemize}
    \item   $[n]=\set{1,\ldots, n}$, and $[n] + k = \{k+1,k+2,...,k+n\}$
    \item Given a non-empty set $S$
    $$\binom{S}{n}$$
    will denote  the set of all subsets  $T\subseteq S$ such that $|T|=n$
\end{itemize}

 \end{notation}

Let $G=(V,E)$ be a graph. We say that $G$ admits a $b$-fold coloring using $c$ colors if there exists an assignment of $b$ distinct colors from the set $[c]$ to $V$ so that whenever $(v,w)\in E$, the set of colors assigned to $v$ is disjoint from that of $w$. 
The $b$-fold chromatic number of a graph $G$, denoted by $\chi_b(G)$, is the minimum number of colors such that $G$ admits a $b$-fold coloring. 
The Kneser graph has vertices given by the sets $\binom{[c]}{b}$, and has an edge between sets $S$ and $T$ if and only if $S\cap T=\emptyset$. We may also understand the $b$-fold chromatic number as a homomorphism into the Kneser graph $K_{c,b}$.  
For the convenience of the reader, we summarize select facts about the $b$-fold chromatic number.   

\begin{enumerate}
    \item If $b=1$, then the $b$-fold chromatic number reduces to the standard chromatic number
    \label{item:bfoldReducesToChromatic}
    
    \item $\chi_b(G)\leq b\chi(G)$\label{item:bfoldChromaticNumberBoundedByChromatic}
    
    \item If $G$ is the complete graph on $n$ vertices then $\chi_b(G)= nb$

    \label{item:completeGraphBfold}

    \item If $a,b\in \N$ then
    $\chi_{a+b}(G)\leq \chi_a{G} +\chi_b(G)$.
    \label{item:bFoldSubadditiveClassical}
 
     \item If $G=H[K]$ is the lexicographic product of graphs (see Section \ref{sec:GraphProducts})  and $b=\chi(K)$,  then  $\chi (G)= \chi_{b}(H)$.\label{item:bfoldlexicographic}
\end{enumerate}

While the $b$-fold chromatic number of a graph $G$ initially appears to be simply $b\cdot \chi(G)$, one can readily verify that for the 5-cycle graph $C_5$, $\chi(G)=3,$ and $\chi_2(G)=5$, illustrating that the inequality in \ref{item:bfoldChromaticNumberBoundedByChromatic} may be strict.  For the purpose of determining bounds on quantum chromatic numbers of products of quantum graphs, we are most interested in item \ref{item:bfoldlexicographic} above, which will be further explored in  Section \ref{sec:lexicographicBound}.

In our present work, motivated by the results found in classical graph theory and extensions of their results to quantum graph theory, our aim is to develop a variant of the $b$-fold chromatic number, suitable for quantum graphs.  To that end, we modify the form of a quantum chromatic number for a quantum graph as described by Brannan, Ganesan, and Harris in \cite{GanesanHarrisQuantumToClassical} to create a quantum $b$-fold chromatic number with quantum input.

\begin{defn}\label{def:BFoldColoring}
Let $\mathcal{G}= (\mathcal{S}, \mathcal{M}, \cB(H))$ be a quantum graph and $b\in \N$. 
For $t\in \{loc, q, qa, qc, \}$, we say that $\mathcal{G}$ is $t$-$b$ colorable (or equivalently, admits a $t$-$b$ fold coloring)  if there exists $c\in \N$,  a von Neumann algebra $\mathcal{N}$,  projections $\set{P_a}_{a=1}^c\in \mathcal{M}\bar\otimes \mathcal{N}$ so that 
\begin{enumerate}
    \item $P_{a_1}P_{a_2}=P_{a_2}P_{a_1} $ for all $a_1,a_2\in \{1,\ldots, c\}$,
    \item $\displaystyle\sum_{\set{a_1,\ldots, a_b}\in \binom{[c]}{b}} P_{a_1}\cdots P_{a_b} = I_{\mathcal{M}\bar\otimes \mathcal{N}}$,\label{item:bFoldPartition}
    \item $P_a(X\otimes I_\mathcal{N})P_a=0$ for all $a\in \set{1,\ldots, c}$, and $X\in\mathcal{S}$.
\end{enumerate}
where  $\dim \mathcal{N} = 1$ if $t = loc$, $\dim\mathcal{N}<\infty$ if $t=q$,  $\mathcal{N}$ embeds into $\mathcal{R}^\omega$ if $t = qa$, or $\mathcal{N}$ is finite if $t=qc$.

We define the $t$-$b$-fold chromatic number of $\mathcal{G}$ to be $$\chi_{b,t}(\mathcal{G})=\min\set{c: \mathcal{G}\text{  admits a }t\text{-}b \text{ fold coloring using }c\text{ colors} }$$
when $\mathcal{G}$ is $t$-$b$-colorable, or $\chi_{b,t}(\mathcal{G})=\infty$ otherwise.
 
\end{defn}
The following result is implicit in Proposition 5.3 of \cite{GanesanHarrisQuantumToClassical}. 
\begin{thm}\label{thm:strategies are unitarily conjugate preserved}
Let $\mathcal{G}= (\mathcal{S}, \mathcal{M},\cB(H))$ and $u\in\mathcal{U}(\cB(H))$ be a unitary.  Then
\[\chi_{b,t}(\mathcal {G})= \chi_{b,t}(u^*\mathcal{S}u, u^*\mathcal{M}u,\cB(H) ) \]
\end{thm}
\begin{proof}
    We will show that there is a bijective correspondence between the $t$-$b$-fold colorings of $\mathcal{G}$ and  $t$-$b$-fold colorings of $(u^*\mathcal{S}u, u^*\mathcal{M}u,\cB(H))$. So see that  $(u^*\mathcal{S}u, u^*\mathcal{M}u,\cB(H))$ is a quantum graph, note that  the equality $u^*M'u=(u^*Mu)'$ is readily verified. Moreover; since $\operatorname{Tr}(Xu^*yu)=\operatorname{Tr}(yuxu^*)$ for all $x,y\in \cB(H)$, $u^*(M')^\perp u =(u^*M'u)^\perp$. Thus, it can be easily checked that $(u^*\mathcal{S}u, u^*\mathcal{M}u,\cB(H))$ is a quantum graph.  
\\
    Now suppose that $\{Q_T\}_{T\in \binom{[c]}{b} }$ are projections in $\mathcal{M}\bar\otimes \mathcal{N}$ summing to $I_{\mathcal{M}\bar{\otimes} \mathcal{N}}$ for an appropriate von Neumann algebra $\mathcal{N}$, then $R_T= (U^*\otimes I_{\mathcal{N}})Q_T(U\otimes \mathcal{N})$. It is clear that $\sum_{T\in \binom{[c]}{b}} R_T= I_{\mathcal{M}\bar{\otimes}\mathcal{N}}$, and that $Q_T(X\otimes I_{\mathcal{N}})Q_S=0$ iff $R_T((u^*Xu)\otimes I_{\mathcal{N}})R_S=0$ for any $X\in \mathcal{S}$ and $S,T\in \binom{[c]}{b}$.  Thus, there is a bijective correspondence between $t$-$b$-fold colorings of $\mathcal{G}$ and $(u^*\mathcal{S}u, u^*\mathcal{M}u,\cB(H))$.
 \end{proof}

To clarify the meaning of quantum $b$-fold chromatic number, we interpret each $P_a$ to be indicating which vertex receives color $a$. The $b$-product of these projections $P_{a_1}...P_{a_b}$ can be understood to represent which vertices are colored by the set of colors $\set{a_1,\ldots, a_b}$. Thus, every vertex being assigned a set of $b$ colors corresponds to the sum of such products being the identity, and the third condition corresponds to preventing adjacent vertices from sharing any colors. The non-local aspect of the game comes from $\mathcal{N}$. Note that the commutativity of the projections is essential to make the coloring order independent.

It is often useful to use an equivalent formulation with projections corresponding to entire sets of colors rather than individual colors. Here we show that this corresponds entirely to Definition \ref{def:BFoldColoring}.

\begin{thm}

    Let $\mathcal{G}=(\mathcal{S}, \mathcal{M}, \cB(H))$ be a quantum graph, $b,c\in \N$ and $t\in \{loc, q, qa,qc\}$. The following are equivalent:
    \begin{enumerate}
        \item $\mathcal{G}$ is $t$-$b$-fold colorable using $c$ colors.\label{item:bfoldcolorable}
        \item There exists a finite von Neumann algebra $\mathcal{N}$ with a faithful normal trace, a family of projections $\set{Q_T}_{T\in \binom{[c]}{b}} \in \mathcal{M}\bar\otimes \mathcal{N} $ such that
            \begin{itemize}
                \item $\sum_{T\in \binom{[c]}{b}} Q_T= I_{\mathcal{M}\bar\otimes \mathcal{N}}$ 
                \item $Q_S(X\otimes I_\mathcal{N})Q_T=0$  for all $S,T\subseteq \binom{[c]}{b}$  such that $S\cap T\neq \emptyset$, and for all $X\in \mathcal{S}$
            \end{itemize}
            where  $\dim \mathcal{N} = 1$ if $t=loc$,  $\dim \mathcal{N}<\infty$ if $t=q$ ,   $\mathcal{N}$ embeds into $\mathcal{R}^\omega$ if $t=qa$, $\mathcal{N}$ is tracial if $t=qc$ .
            \label{item:bfoldcolorableIndexedSets}
        \item There exists a $t$-homomorphism $\Phi: \mathcal{G}\to K_{c,b}$ \label{item:bfoldcororableGraphHomomorphism}
    \end{enumerate}
    \label{Thm:equivBFoldDefs}
\end{thm}
\begin{proof}

First we show that \ref{item:bfoldcolorable}$\implies$\ref{item:bfoldcolorableIndexedSets}:

    Let us assume that $\mathcal{G}$ is $t$-$b$-colorable using $c$ colors. Let $\mathcal{N}$ be a von Neumann algebra and $\set{P}_{a=1}^c\in \mathcal{M}\bar\otimes \mathcal{N}$ be projections as in the definition of a $t$-$b$-fold coloring.  For each $T\in \binom{[c]}{b}$ with $T=\set{a_1,\ldots, a_b}$, let $Q_T=P_{a_1}\cdots P_{a_b} $. By (1) of Definition \ref{def:BFoldColoring} indexing the $Q_T$ by $b$ element subsets is well defined.  
    Further, by (2) of Definition \ref{def:BFoldColoring}, \[\sum_{T\in  \binom{[c]}{b}} Q_T= I_{\mathcal{M}\bar{\otimes } \mathcal{N}}.\]
    Now suppose that $S,T\in \binom{[c]}{b}$ with $a\in S\cap T$. Using the definition of $Q_T$ as well as property (1) of Definition \ref{def:BFoldColoring}, we may write $Q_T=P_aQ_{T'}$ and $Q_S=Q_{S'}P_a$ for some $S',T'\in \binom{[c]}{b-1}$. Thus for any $X\in \mathcal{S}$, (3) of Definition \ref{def:BFoldColoring} implies that $Q_S(X\otimes I_{\mathcal{N}})Q_T=Q_{S'}P_a(X\otimes I_{\mathcal{N}})P_aQ_{T'}=0 $.

    Now, we show \ref{item:bfoldcolorableIndexedSets} $\implies$ \ref{item:bfoldcolorable}:
    
    Assume that \ref{item:bfoldcolorableIndexedSets} holds. For each $a\in [c]$, define 
    \[P_a= \bigvee_{a\in T} Q_T\]
    where $T$ ranges over sets in $\binom{[c]}{b}$.  Equivalently, one can verify that 
    \begin{align*}
      P_a &= \sum_{T} Q_T-\sum_{T_1\neq T_2} Q_{T_1}Q_{T_2} +\sum_{T_1,T_2,T_3, T_i\neq T_j}Q_{T_1}Q_{T_2} Q_{T_3}-\cdots + (-1)^{b} Q_{T_1}Q_{T_2}\cdots Q_{T_b}\\
      &= \sum_{T\in S_a}Q_T
    \end{align*}
    where the sums in the first equality run over distinct sets in $S_a$, and note that the PVM $\{Q_T\}$ is made up of orthogonal projections to obtain the second equality. 
    The collection $\set{P_a}_{a=1}^c$ can be seen to be projections.  
    
    Once again using that the $\{Q_T\}$ are orthogonal projections, we note that $P_aP_b = \sum_{a,b\in T}Q_T$ where $T$ ranges over sets in $\binom{[c]}{b}$, and it is clear that these projections the satisfy the commutation property (1) in Definition \ref{def:BFoldColoring}. 
    Further, by construction we have,  $Q_{a_1,a_2,\ldots, a_b}=P_{a_1}P_{a_2}\cdots P_{a_b}$ so 
    $$ \sum_{\{a_1,...,a_b\}\in \binom{[c]}{b}} P_{a_1}...P_{a_b} = \sum_{T\in \binom{[c]}{b}} Q_T= I_{\mathcal{M}\bar\otimes \mathcal{N}},$$ 
    verifying property (2) in Definition \ref{def:BFoldColoring}.

    Finally, for any $X\in \mathcal{S}$, we have
    \begin{align*}
        P_a(X\otimes I_\mathcal{N})P_a &= P_a\left(\sum_{\{a_1,...,a_b\}\in \binom{[c]}{b}} P_{a_1}...P_{a_b}\right)(X\otimes I_\mathcal{N})\left(\sum_{\{a'_1,...,a'_b\}\in \binom{[c]}{b}} P_{a'_1}...P_{a'_b}\right)P_a\\
        &= \sum_{\{a_1,...,a_b\},\{a'_1,...,a'_b\}\in \binom{[c]}{b}}P_{a_1}\left(Q_{a, a_2,...,a_b}(X\otimes I_\mathcal{N})Q_{a,a'_2,...,a'_b}\right)P_{a'_1}\\
        &= 0
    \end{align*}
    where WLOG we choose $a_1=a$ if $a\in \{a_1,...,a_b\}$ (and the same for $a'_1$), so that the $Q$s are well defined, and the final equality comes from $\{a,a_2,...,a_b\}\cap \{a, a'_2,...,a'_b\} \neq \emptyset$ so $Q_{a, a_2,...,a_b}(X\otimes I_\mathcal{N})Q_{a,a'_2,...,a'_b} = 0$ via our definition of the $Q$s in \ref{item:bfoldcolorableIndexedSets}. 
    
    The equivalence between items \ref{item:bfoldcolorableIndexedSets} and \ref{item:bfoldcororableGraphHomomorphism} can be seen by examining the definition of the Kneser graph $K_{c,b}$ and comparing with \cite[Theorem 5.7, Theorem 5.8]{GanesanHarrisQuantumToClassical}.
\end{proof}

\begin{remark}\label{rmrk:big_products_disappear}
    As a consequence of the second equivalence in Theorem \ref{Thm:equivBFoldDefs} we observe that, given $\{P_a\}_{a\in [c]}$ forming a $b$-fold coloring, the product of $b+1$ distinct projections in this family must vanish.
\end{remark}

Here we verify that this is indeed an extension of the $b$-fold chromatic number by establishing that $\chi_{loc, b}(\mathcal{G})=\chi_b({G})$ where $\mathcal{G}$ is the quantum graph formulation of a graph $G$, and further verifying that we get the expected inequality due to a quantum advantage for $t\neq loc$. 


\begin{prop}\label{prop:PropsofQBFold}
    Let $\mathcal{G}= (\mathcal{S}, \mathcal{M}, M_n)$ be a quantum graph and $b\in\N$.
    \begin{enumerate}
        \item $\chi_{t}(\mathcal{G}) =\chi_{1,t}(\mathcal{G})$
        for all $t\in \{loc, q, qa, qc\}$.
        
        \item Let $\mathcal{G}$ be a quantum graph corresponding a classical graph $G=(V,E)$. 
        Then $\chi_{b,loc}(\mathcal{G}) = \chi_b(G).$

        \item $\chi_{b-1, t}(\cG) < \chi_{b,t}(\cG).$
        \label{prop:reducing_coloring}

    \end{enumerate}
\end{prop}

\begin{proof}
Fix a quantum graph $\mathcal{G}=(\mathcal{S},\mathcal{M} , \mathcal{B}(H))$
\begin{enumerate}
    \item This is immediate from  Definition \ref{def:BFoldColoring} with $b=1$ and  $\chi_{t}(\mathcal{G})$.

    \item

    We will first show we can represent a classical $b$-fold $c$-coloring on $G$ as a  quantum $b$-fold $c$-coloring of $\mathcal{G}$ with projections given in $\mathcal{M}\otimes \C$. 

Since $\mathcal{G}=(\mathcal{S},\mathcal{M}, \cB(\C^n))$ corresponds to a classical graph, $\mathcal{M}$ is the set of diagonal matrices and $\cS = \cS_\cG $.  
For the moment, take $\mathcal{N}=\C$.  For ease of notation,  we will identify $\mathcal{M}$ with $\mathcal{M}\otimes \C$.

Suppose we have a $b$-fold coloring on $G=(V,E)$ using the colors set $[c]$ to color the $n$ vertices of the graph, and let $\phi:V\to \binom{[c]}{b}$ denote the coloring; i.e. $\phi(v)$ is the set of $b$ colors coloring vertex $v$. We can define projections $\{P_a\}_{a=1}^c\in \mathcal{M}$ to be the diagonal matrices given by $$\left(P_a\right)_{vv} = \begin{cases}
1 & \text{ if } a\in \phi(v)\\
0 & \text{ otherwise}
\end{cases}.$$

As $\mathcal{S}_\cG $ is spanned by the matrix units  $E_{v,w}$ where $(v,w)\in E$, it suffices to check that we have
$$P_a(E_{v,w}\otimes 1)P_a = 0$$
This clearly holds as for all $(v,w)\in E$ the equality
$$\left(P_a(E_{v,w})P_a\right)_{xy}  = 0$$
because the LHS can only be nonzero if $x=v, y=w$ and also $a\in \phi(v) \cap \phi(w)$ which would contradict $\phi$ being a $b$-fold coloring since $v\sim w.$

Since the collection of projections $\set{P_a}_{a=1}^c$ are diagonal matrices, they indeed form a collection of pairwise commuting projections. 
Further, as the $P_a$ are diagonal, so are their products and the sum of products of them. Moreover,  we find that
\begin{align*}
    \sum_{S\in \binom{[c]}{b}} (P_{s_1}...P_{s_b})_{vv}
    &= \sum_{S\in \binom{[c]}{b}} (P_{s_1})_{vv}...(P_{s_b})_{vv}\\
    &= 1
\end{align*}
because for a given $S\in \binom{[c]}{b}$ we only have $(P_{s_1})_{vv}...(P_{s_b})_{vv} \neq 0 \implies (P_{s_1})_{vv}...(P_{s_b})_{vv} = 1$ if $s_1,...,s_b\in \phi(v)$ which only occurs for a single set $S$. Thus, $$\sum_{S\in \binom{[c]}{b}} (P_{s_1}...P_{s_b})_{vv} = I_{\mathcal{M}\otimes \C}$$
and so the $P$s satisfy the final condition of definition \ref{def:BFoldColoring}. Therefore, we have found a quantum $b$-fold $c$-coloring on $\mathcal{G}$ with the entanglement removed by $\mathcal{N} = \C$.  
This shows that $\chi_{loc,b}(\mathcal{G})\leq \chi_b(G)$. 

 In order to obtain the reverse inequality we will now show that every quantum $b$-fold $c$-coloring of $\mathcal{G}$ with projections given in $\mathcal{M}\otimes \C$  (hence every local $b$-fold coloring of the form in Definition \ref{def:BFoldColoring}) corresponds to a classical $b$-fold $c$-coloring on $G$.

Suppose we have a quantum $b$-fold $c$-coloring of $\mathcal{G}$ with projections given in $\mathcal{M}\otimes \C$ using the projections $\{P_a\}\in \mathcal{M}\otimes \C$ and their resultant projections $\{Q_S\}$ for $S = \{s_1,...,s_b\}\subset [c]$ as defined in \ref{def:BFoldColoring} and the equivalent formulation in Theorem \ref{Thm:equivBFoldDefs}: \ref{item:bfoldcolorableIndexedSets} (once again, we will often ignore the $\C$ part of these projections). We can construct a resulting $b$-fold $c$-coloring $\phi:V\to \binom{[c]}{b}$ on $G$ as follows: For any $S = \{s_1,...,s_b\}\subset [c]$, let 
$$M_S\otimes \lambda = Q_S(I_\cM \otimes 1)Q_S = Q_S^2 = Q_S$$
where $\lambda \in \C$ and $M_S\in \mathcal{M}$ since $Q_S\in \mathcal{M}\otimes \C$. Then since $\mathcal{M}$ is the set of diagonal matrices we have $M_S$ is a diagonal matrix. If $(M_S)_{vv} \neq 0 (\implies (M_S)_{vv} = 1)$ for any $v\in V$ then let $\phi(v) = S$.

Finally, while we must show that coloring the vertices with the sets of colors as described above is a $b$-fold $c$-coloring of $G$, and this occurs if and only if no adjacent vertices are colored with any of the same colors. We can see that this is true by noticing that for $S=\{s_1,...,s_b\},T=\{t_1,...,t_b\}\subset [c]$ with $S\cap T\neq \emptyset$ we get 
$$Q_S(X\otimes 1)Q_T = 0$$
for any $X\in \cS_\cG $, so for any adjacent vertices $u,v$ we cannot have $\phi(u) = S$ and $\phi(v) = T$.

\item We will prove this by showing that given a $t$-$b$-fold coloring of $\cG$ we can produce a $t$-$(b-1)$-fold coloring with the same number of colors. Let $\{P_a\}_{a=1}^c\subset \cM\bar\otimes\cN$ be a $t$-$b$-fold coloring of $\cG$ with $\cN$ a finite von Neumann algebra corresponding to $t$. Then define $\Tilde{P}_1 = P_1$ and inductively define 

\begin{align*}
    \Tilde{P}_a =& P_a\left(I_{\cM\bar\otimes \cN} - \sum_{n = 1}^{a-1} \Tilde{P}_n\right)\\
    =& P_a - P_a(P_1+\cdots+P_{a-1}) + P_a(P_1P_2 + P_1P_3 +\cdots+ P_{a-2}P_{a-1}) -\cdots\\ &+ (-1)^{a-1}P_1P_2\cdots P_a.
\end{align*}

It is obvious that these form a mutually orthogonal set of projections, and it is easily verified that $\{\Tilde{P}_a\}_{a=1}^c$ is a $1$-fold coloring using $c$ colors as $\Tilde{P}_a$ is a subprojection of $P_a$ and these are specifically designed via inclusion-exclusion to be a partition of unity. Using these, we define the $(b-1)$-fold coloring on $c$ colors via
\begin{align*}
    P_a' &= P_a - \Tilde{P}_a = P_a \sum_{n=1}^{a-1}\Tilde{P}_n
\end{align*}
for $a>1$ and $P_1' = 0$. 

We verify this is a coloring by checking the following: the projections commute, the coloring condition is satisfied, and they sum to the identity.
As the sum and multiples of commuting projections, the $P_a'$ commute with each other. Further, the coloring condition is satisfied as $P_a'(X\otimes I_\cN) P_a' = \left(\sum_{n=1}^{a-1}\Tilde{P}_n\right)P_a(X\otimes I_\cN)P_a\left(\sum_{n=1}^{a-1}\Tilde{P}_n\right) = 0$ for $X\in \cS$.
Now, for each set $\set{a_1,\ldots, a_{b-1}}\in \binom{[c]}{b-1}$, we will assume (without loss of generality) that $a_1$ is the minimal element. Then,
\begin{align*}
    \sum_{\{a_1,...,a_{b-1}\}\in \binom{[c]}{b-1}} P_{a_1}'...P_{a_{b-1}}' &= \sum_{\{a_1,...,a_{b-1}\}\in \binom{[c]}{b-1}} P_{a_1}...P_{a_{b-1}}\left(\sum_{n=1}^{a_1-1}\Tilde{P}_n\right)  ... \left(\sum_{m=1}^{a_b-1}\Tilde{P}_m\right)\\
    &= \sum_{\{a_1,...,a_{b-1}\}\in \binom{[c]}{b-1}} P_{a_1}...P_{a_{b-1}}\left(\sum_{n=1}^{a_1-1}\Tilde{P}_n\right)\\
    &= \sum_{\{a_1,...,a_{b-1}\}\in \binom{[c]}{b-1}}\sum_{n=1}^{a_1-1}P_nP_{a_1}...P_{a_{b-1}}\\
    &= \sum_{T\in \binom{[c]}{b}}\prod_{k\in T}P_k = I_{\cM\bar\otimes\cN}
\end{align*}
where the second equality comes from the $\Tilde{P}_a$ being mutually orthogonal, and the third equality holds because all except the first term in the definition of $\Tilde{P}_a $ will vanish in the product due to Remark \ref{rmrk:big_products_disappear}. Hence, $\{P_a'\}_{a=2}^c$ form a $b-1$-coloring (taking $a>1$ since $P_1'=0$) with $c-1$ colors.
\end{enumerate}
\end{proof} 

We now show an essential property of $b$-fold numbers carries over to our new definition, Namely,  property \ref{item:bFoldSubadditiveClassical}.

\begin{prop}
$\chi_{b_1+b_2,t}(\mathcal{G})\leq \chi_{b_1,t}(\mathcal{G}) +\chi_{b_2,t}(\mathcal{G})$
for all  $b_1,b_2\in \N$, $t\in \{ loc, q,qa,qc\}$ and quantum graphs $\mathcal{G}$.\label{prop:bFoldSubadditive}
\end{prop}

We briefly pause to remark on the meaning of Proposition \ref{prop:bFoldSubadditive}. The  subadditivity of the quantum (or in general $t$) b-fold chromatic number shown in Proposition \ref{prop:bFoldSubadditive}  and a classical result of analysis show that,
$$\lim_{b\to\infty}  \frac{\chi_{b,t}(\mathcal{G})}{b} =\inf\set{\frac{\chi_{b,t}(\mathcal{G})}{b}:b\in \N}.$$
Hence, the limit of the ratio above the exists and is finite. The analogous limit in the setting of above in the classical case is called the \emph{fractional chromatic number of a graph}; this inspires us to define the following
$\chi_{f,q}(\mathcal{G}):=\lim_{b\to\infty}  \frac{\chi_{b,q}(\mathcal{G})}{b}$ (with an analogous definition for $t\neq q$). It is an interesting question to determine what relationship, if any, exists between this value and other related quantum graph invariants such as the tracial rank and projective rank discussed in \cite{fractional}. In particular, it was noted that an irrational value of the projective rank would provide a counterexample to Tsirelson's conjecture, so it may be of interest to determine if a similar relation holds for $\chi_{f,q}$.

In the sequel, we will adopt the following notation. 
\begin{notation}\label{notaiton:Permulted}
    Let  $\theta:B({H})\bar\otimes B(H_{\mathcal{N}_2})\bar\otimes B(H_{ \mathcal{N}_1})\to B({H})\bar\otimes B(H_{\mathcal{N}_1})\bar\otimes B(H_{ \mathcal{N}_2})$ is the flip isomorphism. If $Q\in \mathcal{M}\bar\otimes \mathcal{N}_2$ and $X\in \mathcal{N}_1$, then we will denote $\theta(Q\otimes X)\in \mathcal{M}\bar\otimes \mathcal{N}_1\bar\otimes \mathcal{N}_2$ by $Q\odot X$. In the sequel, we will refer to this as the \emph{permuted tensor product}.
\end{notation}

\begin{proof}[Proof of Proposition \ref{prop:bFoldSubadditive}] 
Let $\cG = (\cS, \cM, \cB(H))$. If $\chi_{t,b_1}(\mathcal{G})$ or $\chi_{t,b_2}(\mathcal{G})$ are infinite then the result follows trivially.  Thus we will  assume that both values are finite.


Let $\{Q_S^1\}_{S\in \binom{[c]}{b_1}}\in \mathcal{M}\bar\otimes\mathcal{N}_1$ and $\{Q_T^2\}_{T\in \binom{[d]}{b_2}}\in \mathcal{M}\bar\otimes\mathcal{N}_2$ be $t$-$b_1$ and $t$-$b_2$ colorings, respectively,   in sense of Theorem \ref{Thm:equivBFoldDefs}. 
Set $\mathcal{N}=\mathcal{N}_1\bar\otimes \mathcal{N}_2$; observe that $\mathcal{N}$ is $\C$, finite dimensional, $\mathcal{R}^\omega$ embeddable, or finite whenever the same holds for both $\mathcal{N}_1$ and $\mathcal{N}_2$.
Identify
$\{Q_S^1\}_{S\in \binom{[c]}{b_1}} $ with $\{Q_S^1\otimes I_{\mathcal{N}_2}\}_{S\in \binom{[c]}{b_1}}$  and $\{Q_T^2\}_{T\in \binom{[d]}{b_2}}$ with $\{Q_T^2\odot I_{\mathcal{N}_1}\}_{T\in \binom{[d]}{b_2}}$ in $\mathcal{M}\bar\otimes \mathcal{N}_1\bar\otimes \mathcal{N}_2$, where $\odot $ is the permuted tensor product. 
We summarize the relevant facts: $
\sum_{S\in\binom{[c]}{b_1}} Q_{S}^i= I_{\mathcal{M}\otimes \mathcal{N}_1\bar\otimes N_2}$;
and for each $i\in \{1,2\}$, $Q_{S_1}^i(X\otimes I_{\mathcal{N}_1}\otimes I_{\mathcal{N}_2})Q_{S_2}^i=0$ whenever  $X\in \mathcal{S}$ and $S_1, S_2\in \binom{[c]}{b_i} $ have non-empty intersection.

Let $A\subseteq \binom{[c+d]}{b_1+b_2}$ denote the collection of all elements of $\binom{[c+d]}{b_1+b_2}$ with contain exactly $b_1$ elements less than or equal to $c$.   $A$ can be written as the following disjoint union:
\[A=\bigcup_{S_2\in \binom{[d]}{b_2}}\left\{S_1\cup (S_2+c):S_1\in \binom{[c]}{b_1} \right\}; \]
observe further that the sets $S_1\cup (S_2+c)$ above are disjoint as well.
With this description in hand, for each $S\in \binom{[c+d]}{b_1+b_2}$, define
\[
Q_{S}= \begin{cases}
    Q^1_{S_1}\wedge Q^2_{S_2} & \text{ if }S=S_1\cup (S_2+c)\in A, \\
    0 & \text{ otherwise}.
\end{cases}
\]
Whenever $Q_{S}= Q^1_{S_1}\wedge Q_{S_2}^2$, it follows that $Q_SQ^1_{S_1}=Q^1_{S_1}Q_S=Q_{S}$ and $Q_SQ_{S_2}^2= Q_{S_2}^2Q_S=Q_{S}$.
One can readily verify that the projections $\{Q_S\}_{S\in \binom{[c+d]}{b_1+b_2}}$ are an orthogonal family of projections in $\mathcal{M}\bar\otimes \mathcal{N}$.

Now we compute:
\begin{align*}
\sum_{S\in \binom{[c+d]}{b_1+b_2}} Q_S= \sum_{S\in A}Q_S= \sum_{S_2\in \binom{[d]}{b_2}} \sum_{S_1\in \binom{[c]}{b_1}}  Q_{S_1}^1\wedge Q_{S_2}^2 
=& \sum_{S_2\in \binom{[d]}{b_2}}\left( \sum_{S_1\in \binom{[c]}{b_1}}  Q_{S_1}^1\right) \wedge Q_{S_2}^2 \\
=& \sum_{S_2\in \binom{[d]}{b_2}} I_{\mathcal{M}\bar\otimes \mathcal{N}}\wedge Q_{S_2}^2 \\
=& \sum_{S_2\in \binom{[d]}{b_2}}  Q_{S_2}^2\\
= & I_{\mathcal{M}\bar\otimes \mathcal{N}}
\end{align*}
 Now, if $X\in \mathcal{S}$ and $T, S\in A$ have non-empty intersection, \[Q_T(X\otimes I_{\mathcal{N}_1\otimes \mathcal{N}_2})Q_S=Q_T[Q^1_{T_1}(X\otimes I_{\mathcal{N}_1\otimes \mathcal{N}_2})Q^1_{S_1}]Q_S=Q_T[Q^2_{T_1}(X\otimes I_{\mathcal{N}_1\otimes \mathcal{N}_2})Q^2_{T_1}]Q_S=0.\]

\end{proof}
\begin{remark}
    The argument above is greatly simplified if for example $t=loc$ and $\mathcal{G}=(\mathcal{S},\mathcal{M}, B(H))$ is such that $\mathcal{M}$ is abelian. In this case, $\theta$ is the identity and  $Q_T\wedge \theta(Q_S)=Q_TQ_S=Q_SQ_T$. 
\end{remark}

The previous result gives us the following useful  bound corresponding to property \ref{item:bfoldChromaticNumberBoundedByChromatic} for the classical $b$-fold chromatic number.

\begin{cor}
    For any quantum graph $\mathcal{G}$ and $t\in \{loc, q, qa, qc\}$, we have 
    $$\chi_{b,t}(\mathcal{G})\leq b\chi_t(\mathcal{G})$$
    \label{item:bFoldVersusChromatic}
\end{cor}

\begin{proof}
    Apply Prop. \ref{prop:bFoldSubadditive} using $\chi_{1,t}(\mathcal{G}) = \chi_t(\mathcal{G}).$
\end{proof}

For fully quantum graphs it is often impossible to obtain finite local colorings. In particular, for complete quantum graphs, i.e., graphs of the form $(M_n\cap (\cM')^\perp, \mathcal{M}, M_n)$, the following corollary holds.

\begin{cor} Let $\cG = (M_n\cap (\cM')^\perp, \cM, \cB(H))$ be a complete quantum graph. Then 
    $\chi_{b,loc}(\mathcal{G})<\infty$ if and only  $\mathcal{M}$ is abelian. 
\end{cor}
\begin{proof}
If $\cM$ is abelian then \cite[Theorem 6.11]{GanesanHarrisQuantumToClassical} gives us that $\chi_{loc}(\cG)<\infty$ and with Corollary \ref{item:bFoldVersusChromatic} this implies $\chi_{b,loc}(\cG)\leq b\chi_{loc}(\cG)<\infty$. If $\cM$ is not abelian then \cite[Theorem 6.11]{GanesanHarrisQuantumToClassical} gives us that $\chi_{loc}(\cG)$ is not finite, so Item \ref{prop:reducing_coloring} of Proposition \ref{prop:PropsofQBFold} gives us $\chi_{b, loc}(\cG)$ is not finite either.
\end{proof}

Finally, we would like to show a proposition equivalent to property \ref{item:completeGraphBfold} for the classical $b$-fold chromatic number. To do so, we will closely follow the proof of \cite[Theorem 6.9]{GanesanHarrisQuantumToClassical}, with the main change being the factor of $b$ that come from 
$$\sum_{a=1}^c P_a = \sum_{a=1}^c\sum_{T\in \binom{[c]}{b}, a\in T} Q_T = \sum_{T\in \binom{[c]}{b}} \sum_{a\in T} Q_T  =b\sum_{T\in\binom{[c]}{b}}Q_T = bI_{\mathcal{M}\bar\otimes\mathcal{N}}$$
where we let $\{P_a\}_{a=1}^c$ be a $b$-fold coloring as described in Definition \ref{def:BFoldColoring} and $\{Q_T\}$ the corresponding PVM from Theorem \ref{item:bfoldcolorableIndexedSets}.
To prove the desired proposition, we first establish a useful lemma about complete quantum graphs.
The following lemma is a $b$-fold coloring analog of \cite[Lemma 6.8]{GanesanHarrisQuantumToClassical}, and is proved very similarly aside from making use of the above formula. We provide the proof for completeness.

\begin{lem}
    Let $d,k\in \N$ and let $n=dk$. Consider the quantum graph $(M_n\cap (\cM')^\perp, \mathcal{M}, M_n)$ with $\mathcal{M} = \C I_d \otimes \M_k$. Let $\mathcal{N}$ be a finite von Neumann algebra with a faithful normal trace, and let $\{P_a\}_a\in [c]\in \mathcal{M}\bar\otimes\mathcal{N}$ be commuting projections satisfying 
    $$P_a(X\otimes I_\mathcal{N})P_a = 0$$ for all $X\in M_n\cap (\cM')^\perp$,
    and $\sum_{\{a_1,..,a_b\}\in \binom{[c]}{b}}P_{a_1}...P_{a_b} = I_{\mathcal{M}\bar\otimes\mathcal{N}}.$

    Then for each $a\in[c]$, the element $R_a = \frac{k}{d\dim\mathcal{M}}(Tr_{dk}\otimes id_\mathcal{N})(P_a)$ is a self-adjoint idempotent in $\mathcal{N}$, and $\sum_T R_T = bk^2I_\mathcal{N}$.

    \label{lemma:completeGraphEqualitySpecialCase}
\end{lem}

\begin{proof}
    We have $\mathcal{M}' = M_d\otimes \C I_k$, and $n=dk$. Let $1\leq v,w\leq k$ with $x\neq y$, and let $1\leq i,j\leq d$. Then if $E_{ij}$ is the matrix with 1 at the $ij$th entry and zero everywhere else, we have $E_{ij}\otimes (E_{vv} - E_{ww})\in (\mathcal{M}')^\perp$, and we obtain our first coloring condition: 
       $ P_a(E_{ij}\otimes (E_{vv} - E_{ww}))P_a = 0,\, \forall a\in[c].\label{eq:FirstColoringCondition}$
  
    Similarly, $E_{ij}\otimes E_{vw}\in (\mathcal{M}')^\perp$, and we obtain our second coloring condition:
    $P_a(E_{ij}\otimes E_{vw})P_a = 0,\, \forall a\in[c].$  
Since $P_a \in \mathcal{M}\bar\otimes \mathcal{N} = I_d\otimes M_k\otimes \mathcal{N}$ we have $P_a = \sum_{p,q=1}^k\sum_{x=1}^d E_{xx}\otimes E_{pq}\otimes P_{a,x,pq}$ while satisfying $P_{a,x,pq} = P_{a,y,pq}$ for any $1\leq x,y\leq d$. Thus, we may set $P_{a,x,pq} = P_{a,pq}$.

    The  first coloring condition then implies that $$\sum_{p,q=1}^k E_{ij}\otimes E_{pq}\otimes P_{a,pv}P_{a,wq}=0,$$ so we find $P_{a,pv}P_{a,wq} = 0$. From the second coloring condition above,  $P_{a,pv}P_{a,vq} = P_{a,pw}P_{a,wq}$. As the $P_a$'s are projections, we have $P_{a,pq} = \sum_{v=1}^k P_{a,pv}P_{a,vq} = k P_{a,pv}P_{a,vq}$ for all $p,q$, and in particular $P_{a,vv} = k(P_{a,vv})^2$. Thus, we easily see that $kP_{a,vv}$ is self-adjoint and idempotent. Similarly, since $P_{a,pv}P_{a,wq}=0$ if $v\neq w$ we get $P_{a,vv}P_{a,ww} = 0$. Thus, $\{kP_{a,vv}\}_{v\in[k]}$ is a collection of mutually orthogonal projections in $\mathcal{N}$.

    Finally, setting $R_a = \sum_{v=1}^k kP_{a,vv}$ for all $a\in [c]$ we see that $R_a$ is a projection and 
    $$\sum_{a=1}^c R_a = \sum_{a=1}^c\sum_{v=1}^k kP_{a,vv} = \sum_{v=1}^k b kI_\mathcal{N} = b k^2I_\mathcal{N}$$
    
\end{proof}

Finally, we extend to a general complete quantum graph using a direct sum decomposition and the above lemma to obtain $\chi_{b,t}(\mathcal{G})\geq b\dim \mathcal{M}$. Combining this with the more easily obtained reverse inequality from Corollary \ref{item:bFoldVersusChromatic}  will demonstrate that $\chi_{b,t}(\mathcal{G})= b\dim \mathcal{M}$. The proof is a modification of the proofs of \cite[Theorem 6.9 and Proposition 6.3]{GanesanHarrisQuantumToClassical} for the $b$-fold setting.

\begin{thm}
    For the complete quantum graph $\mathcal{G} = (M_n\cap (\cM')^\perp, \mathcal{M}, M_n)$ and $t \in \{q, qa, qc\}$ we have 
    $$\chi_{b,t}(\mathcal{G}) = b\dim \mathcal{M} $$
    \label{thm:completeGraph}
\end{thm}
\begin{remark}
    While we are not investigating the hereditary colorings of quantum graphs in this manuscript, we remark that the following proof can be generalized to that setting with no change.
\end{remark}
\begin{proof}
    
    From Corollary \ref{item:bFoldVersusChromatic} and the result of \cite{GanesanHarrisQuantumToClassical} that $\chi_t(\mathcal{G}) = \dim \mathcal{M}$, we easily see that $\chi_{b,t}(\mathcal{G})\leq b\dim \mathcal{M}$. 
    To prove the opposite inequality, suppose we may write $\mathcal{M} = \oplus_{r=1}^m \C I_{n_r}\otimes M_{k_r}$. Then we get $\mathcal{M}' = \oplus_{r=1}^mM_{n_r}\otimes \C I_{k_r}$.
    To apply Lemma \ref{lemma:completeGraphEqualitySpecialCase},  define $\mathcal{E}_r = 0\oplus...\oplus I_{n_r}\otimes I_{k_r}\oplus0\oplus...\oplus0\in \mathcal{M}'\cap\mathcal{M}.$ Suppose that $\{P_a\}_{a\in [c]}$ is a quantum commuting $b$-fold coloring as in Definition \ref{def:BFoldColoring}, and $\{Q_T\}_{T\in \binom{[c]}{b}}$ the corresponding projections as in Theorem \ref{item:bfoldcolorableIndexedSets}. Then, defining $\Tilde{Q}_T = (\mathcal{E}_r\otimes I_\mathcal{N})Q_T(\mathcal{E}_r\otimes I_\mathcal{N})\in (\mathcal{E}_r\mathcal{M}\mathcal{E}_r)\bar\otimes \mathcal{N}$, we obtain a family of mutually orthogonal projections whose sum is $\mathcal{E}_r$ (the identity in $\mathcal{E}_r\mathcal{M}\mathcal{E}_r)$. Since $\mathcal{E}_r$ is central in $\mathcal{M}$, we get $(\mathcal{E}_r\mathcal{M}\mathcal{E}_r)' = \mathcal{E}_r\mathcal{M}'\mathcal{E}_r$, while we also have $\mathcal{E}_rM_n\mathcal{E}_r = M_{n_rk_r}$.

    It is clear that $X\in \mathcal{B}(\mathcal{E}_r\C^n)\cap (\mathcal{E}_r\mathcal{M}'\mathcal{E}_r)^\perp$ if and only if $X = \mathcal{E}_rX\mathcal{E}_r$ and $X\in M_n\cap (\cM')^\perp$. Thus, for such $X$ and $S,T \in \binom{[c]}{b}$ we have 
    $$\Tilde{Q}_S(X\otimes I_\mathcal{N})\Tilde{Q}_T = (\mathcal{E}_r\otimes I_\mathcal{N})Q_S(\mathcal{E}_rX\mathcal{E}_r\otimes I_\mathcal{N})Q_T(\mathcal{E}_r\otimes I_\mathcal{N}) = 0$$
    if $S\cap T \neq \emptyset$ because $X = \mathcal{E}_rX\mathcal{E}_r$ and $X\in (\mathcal{M}')^\perp$. Thus, $\{\Tilde{Q}_T\}_{T\in \binom{[c]}{b}}$ gives a quantum commuting $b$-fold coloring of $(M_{n_rk_r}, \mathcal{E}_r\mathcal{M}\mathcal{E}_r,M_{n_rk_r})$.

    Since $\mathcal{E}_r\mathcal{M}\mathcal{E}_r = \C I_{n_r}\otimes M_{k_r}$ we may now use Lemma \ref{lemma:completeGraphEqualitySpecialCase} with $R_a^{(r)} = \frac{k_r}{n_r}(Tr_{n_rk_r}\otimes id_\mathcal{N})(\Tilde{P}_a)$ the self-adjoint idempotent in $\mathcal{N}$ where $\Tilde{P}_a$ are the projections from Definition \ref{def:BFoldColoring} corresponding to the $\Tilde{Q}_T$ as described in Theorem \ref{Thm:equivBFoldDefs}. (Note that we may easily check that $\Tilde{P}_a = (\mathcal{E}_r\otimes I_\mathcal{N})P_a(\mathcal{E}_r\otimes I_\mathcal{N})$, as one would expect, using the correspondence laid out in Theorem \ref{Thm:equivBFoldDefs}). Moreover, we have $\sum_{a=1}^c R_a^{(r)} = bk_r^2I_\mathcal{N}$.

    We claim that $R_a^{(r)}R_a^{(s)} = 0$ if $r\neq s$. To see this it suffices to show that whenever $P_{a,xx}$ is a block from $(\mathcal{E}_r\mathcal{M}\mathcal{E}_r\bar\otimes \mathcal{N})$ and $P_{a,yy}$ a block from $(\mathcal{E}_s\mathcal{M}\mathcal{E}_s\bar\otimes \mathcal{N})$ then $P_{a,xx}P_{a,yy} = 0$. If $x$ and $y$ are chosen this way then $E_{xy}\in M_n$ satisfies $\mathcal{E}_rE_{xy}\mathcal{E}_s = E_{xy}$ and $\mathcal{E}_pE_{xy}\mathcal{E}_q =0$ for all other pairs $(p,q)$. Then we can see that $E_{xy}\in (\mathcal{M}')^\perp$, so we must have $P_a(E_{xy}\otimes I_\mathcal{N})P_a = 0$. Thus, considering the $(x,y)$ block of this equation we get $P_{a,xx}P_{a,yy} = 0$ as needed.

    Since $\{R_a^{(r)}\}_{r\in[m]}$ is a mutually orthogonal family of projections in $\mathcal{N}$, the element $R_a = \sum_{r\in[m]}R_a^{(r)}$ is a self-adjoint idempotent in $\mathcal{N}$ for each $a\in[c]$. Considering blocks, we can see that 
    $$\sum_{a\in[c]}R_a = \sum_{a=1}^c\sum_{r=1}^m R_a^{(r)} = \sum_{r=1}^m bk_r^2I_\mathcal{N} = b\dim\mathcal{M}I_\mathcal{N}.$$
    $R_a$ is a self-adjoint idempotent, so  $I_\mathcal{N} - R_a$ must be as well. The sum of these projections is 
    $$\sum_{a=1}^c(I_\mathcal{N}-R_a) = (c-b\dim\mathcal{M})I_\mathcal{N},$$
    and as this must be positive we get $c\geq b\dim\mathcal{M}$. Thus, $\chi_{b,qc}(\mathcal{G})\geq b\dim\mathcal{M}$; furthermore, as $\chi_{b,t}\geq \chi_{qc,t}$ 
 for all $t\in \{q, qa, qc\}$,  we may conclude the result holds when $\mathcal{M}= \oplus_{r=1}^m \C I_{n_r}\otimes M_{k_r}$.

    To complete the proof note that for any  finite dimensional von Neumann algebra $\mathcal{M}$ there is a unitary $U\in M_n$, $m\in \mathbb{N}$, $n_1,\ldots, n_m$ and $k_1,\ldots, k_m$  such that $U^*\mathcal{M}U = \oplus_{r=1}^m \C I_{n_r}\otimes M_{k_r}$, as was shown in \cite[Proposition 6.3]{GanesanHarrisQuantumToClassical}. Then Theorem \ref{thm:strategies are unitarily conjugate preserved} gives 
    $$\chi_{b,t}(M_n\cap (\cM')^\perp,\mathcal{M},M_n) = \chi_{b,t}\left(M_n\cap (\cM')^\perp, \oplus_{r=1}^m \C I_{n_r}\otimes M_{k_r}, M_n\right),$$
    which completes the theorem.
\end{proof}

\begin{remark}
    Note that we can easily see that the above is not true for local $b$-fold colorings. In \cite[Theorem 6.11]{GanesanHarrisQuantumToClassical} it was shown that if $\mathcal{M}$ is not abelian then the local chromatic number is infinite. Thus, since the $b$-fold chromatic number is clearly at least that of the chromatic number, there are not any finite local $b$-fold colorings of a complete fully quantum graph, let alone ones using $b\dim \mathcal{M}$ colors.
\end{remark}

Theorem \ref{thm:completeGraph} is especially powerful considering all quantum graphs are subgraphs of a complete quantum graph, and we have the following $b$-fold analog of \cite[Poposition 6.4]{GanesanHarrisQuantumToClassical}.

\begin{prop}
    If $(\cS, \cM, \cB(H))$ and $(\cT, \cM, \cB(H))$ are quantum graphs with $\cS\subset \cT$, then
    $$\chi_{b,t}(\cS, \cM, \cB(H))\leq \chi_{b,t}(\cT, \cM,\cB(H))$$
    for $t\in \{loc, q, qa, qc\}.$
\end{prop}
\begin{proof}
    This proof proceeds analogously to \cite[Proposition 6.4]{GanesanHarrisQuantumToClassical}.
    \\
    First, if $(\cT, \cM,\cB(H))$ has no $t$-$b$-fold coloring then $\chi_{b,t}(\cT, \cM,\cB(H)) = \infty$ and the result follows trivially. Now, assume that $\{Q_T\}_{T\in \binom{[c]}{b}}$ is a PVM corresponding to a $t$-$b$-fold coloring using $c$ colors. Then evidently they also satisfy the coloring condition for $X\in \cS\subset \cT$, so this is also a $t$-$b$-fold coloring for $(\cS, \cM, \cB(H))$. 
\end{proof}

From the above proposition and Theorem \ref{thm:completeGraph} we have the following corollary. 
\begin{cor}
     For $t\in \{q,qa,qc\}$ and a quantum graph $\cG$, we have 
     $$\chi_{b,t}(\cG)\leq b\dim\cM.$$
     Hence, the $t$-$b$-fold chromatic number of $\cG$ is finite.
\end{cor}

\section{Chromatic Numbers and Products of Quantum Graphs}

\subsection{Bounds for  Cartesian and Categorical Products}\label{sec:BoundsForCategoricalAndCartesianProducts}
Here we show that the analysis given by Kim and Mehta in \cite{graph_products} of the completely bounded minimal chromatic number extends to the quantum chromatic numbers of quantum graph products for the categorical and Cartesian products.
 To do this we will heavily use the transitivity of quantum graph homomorphisms as shown in \cite[Lemma 2.5]{quantum_homo_MANCINSKA2016228} for classical graphs. This immediately  generalizes to quantum graphs because the composition of CPTP maps is a CPTP map.

\begin{prop}\label{transitive_homo}For $t\in \{loc, q, qa, qc\}$, $t$ quantum graph homomorphisms are transitive: let $\cG, \cH$ and $\cK$ be quantum graphs. If $\cG\to_t \cH$ and $\cH\to_t\cK$ then $\cH\to_t\cK$.
\end{prop}

From this we immediately get the following corollary:

\begin{cor}\label{cor:q_color_transitive}
Let $\cG$ and $\cH$ be quantum graphs such that $\cG\to_t \cH$. Then $\chi_t(\cG)\leq \chi_t(\cH)$ for $t\in \{loc, q, qa, qc\}$
\end{cor}
\begin{proof}
Suppose that $n = \chi_t(\cH)$. Then $\cH\to_t \cK_n$ where $\cK_n$ is the associated quantum graph of the classical complete graph $K_n$. We have $\cG\to_t \cH$ so by Prop \ref{transitive_homo} we must get $\cG\to_t \cK_n$ which gives us $\chi_t(\cG)\leq n$.
\end{proof}

Our first extension of Kim and Mehta's work is our variaint of \cite[Theorem 40]{graph_products}, Sabidussi's theorem for quantum graphs. 

\begin{thm}
[Sabidussi's theorem]\label{thm:BoundCartesianProduct}
For $\cG$ and $\cH$ quantum graphs, $$\max\{\chi_t(\mathcal{G}), \chi_t(\mathcal{H})\} \leq  \chi_t(\mathcal{G}\square\mathcal{H})$$ for $t\in \{loc, q, qa, qc\}.$
\end{thm}
\begin{proof}
Let $\cG = \left(\cS_\cG, \cM_\cG, \cB(H_\cG)\right)$, $\cH = \left(\cS_\cH, \cM_\cH, \cB(H_\cH)\right)$, and $\mathcal{G}\square\cH = \left(\cS_{\cG\square \cH}, \cM_{\cG\square \cH}, \cB(H_{\cG\square \cH})\right)$. We will show that $\cG\to_{loc}\cG\square\cH$.
To do this, define a map $\Phi: \cM_\cG \otimes \C \to \cM_{G\square H}$ by $\Phi(X\otimes \lambda) = \lambda X\otimes I_\cH$ where $I_\cH$ is the identity in $\cM_\cH$, so $tr(I_\cH)=1$. This map has been constructed to preserve the trace, and it is clearly completely positive as it is given by the Kraus operator $K = I_\cG  \otimes I_\cH$, i.e.~$\Phi(\cdot) = K(\cdot) K$. 


Applying this map to $\cS_\cG \otimes I_\cN$ where we take $\cN = \C$, we get $\Phi(\cS_\cG  \otimes 1) = \cS_\cG \otimes I_\cH$. This is a subspace of $\cS_\cG \otimes \cM_\cH'$ since $I_\cH\in \cM_\cH'$. A similarly quick calculation shows that applying the map to $\cM_\cG '\otimes I_\cN$ gives a subspace of 
$\cM_\cG '\otimes \cM_\cH' \subset \cM_{G\square H}'$. Thus, $\Phi$ defines a local quantum graph homomorphism $\cG\to_{loc}\cG\square\cH $, which implies $ \cG\to_t\cG\square\cH$.

The exact same argument applies to show that $\cH\to_{t} \cG\square \cH$,  so by Corollary \ref{cor:q_color_transitive}, we have the desired result.
\end{proof}

We now generalize the results for Hedetniemi's inequality to the setting of quantum graphs. The result follows similarly to the proof of Proposition 46 in \cite{graph_products}.

\begin{thm}[{Hedetniemi's inequality}]\label{thm:BoundCategoricalProduct} Let $\cG$ and $\cH$ be quantum graphs. Then
$$\chi_t(\mathcal{G}\times\mathcal{H}) \leq \min\{\chi_t(\mathcal{G}), \chi_t(\mathcal{H})\}$$
for $t\in \{loc, q, qa, qc\}.$
\end{thm}
\begin{proof}
Let $\cG = \left(\cS_\cG, \cM_\cG, \cB(H_\cG)\right)$, $\cH = \left(\cS_\cH, \cM_\cH, \cB(H_\cH)\right)$, and $\mathcal{G}\times \cH = \left(\cS_{\cG\times \cH}, \cM_{\cG\times \cH}, \cB(H_{\cG\times \cH})\right)$.
We will show that $\cG\times\cH\to_{loc} \cG\implies \cG\times \cH\to_t \cG$. An analogous argument will show that $\cG\times\cH\to_t \cH$, and then Corollary \ref{cor:q_color_transitive} gives the desired result.

Define a map $\Phi:\cM_{\cG\times \cH} \otimes \C \to \cM_\cG $ via a partial trace so that we have $\Phi(X\otimes \lambda) = tr_{\cG}(X\otimes \lambda)$. One should note that the map is constructed to be trace-preserving. Further, we may write this map using Krauss operators as $\Phi(\cdot) = \sum_j K_j(\cdot) K_j^*$ via writing $K_j = I_\cG \otimes e_j^* \otimes 1$ where $e_j\in H_\cH\otimes \C$ is the $j$-th unit vector in the standard basis and $I_\cG$ is the identity in $\cM_\cG$. Thus, the partial trace is also completely positive and defines a quantum channel.

Finally, we must check that this CPTP map satisfies the requirements of quantum graph homomorphisms. First, we have $$K_i(\cS_\cG \otimes \cS_\cH \otimes 1) K_j^* = \cS_\cG  \otimes e_i^* (\cS_\cH\otimes 1) e_j  = \cS_\cG  \otimes \C \cong \cS_\cG $$ as desired. Second, we find 
$$K_i (\cM_\cG '\bar\otimes \cM_\cH'\otimes 1) K_j^* = \cM_\cG '\bar\otimes e_i^*(\cM_\cH'\otimes 1)e_j = \cM_\cG '\otimes \C \cong \cM_\cG '.$$
Thus we have $\cG\times \cH \to_t \cG$ as desired.
\end{proof}

\begin{remark}
    In the fully classical setting of classical graphs and $t=loc$, the conjecture that $\chi(G_1\times G_2) = \min(\chi(G_1), \chi(G_2))$ was known as Hedetniemi's conjecture. However, a counterexample to this was found in \cite{ClassicalHedetniemisCounterexamples}, so at least for $t=loc$ the inequality in Theorem \ref{thm:BoundCategoricalProduct} cannot be an equality in general.
\end{remark}

  \subsection{Bounds for Lexicographic Products}\label{sec:lexicographicBound}
The present section is devoted to establishing the bounds on the chromatic numbers for the lexicographic product of quantum graphs as described in the introduction.  In \cite{fractional}, the authors show    $\xi_{tr}(G_1[G_2]) \leq \xi_{tr}(G_1)\xi_{tr}(G_2)$ where $\xi_{tr}$ is the tracial rank.  We are interested in addressing bounds on the quantum chromatic number of $G_1[G_2]$ directly to possibly provide tighter bounds.  To motivate the following result, we remind the reader of the deep connection between the $b$-fold chromatic number and the lexicographic product:  if $G_1$ and $G_2$  are classical graphs and $b=\chi(G_2)$, then $\chi(G_1[G_2])=\chi_b(G_1)$.  The proof of this is standard and we refer the reader to \cite[Theorem 26.3]{HandbookGraphProducts}. 

We are now ready to prove Theorem \ref{thm:Main_Lex_prod}.
\begin{thm}\label{thm:lexicographicBoundedByProduct}
For quantum graphs $\mathcal{G} = (\cS_\cG , \cM_\cG , \cB(H_\cG ))$ and $\mathcal{H} = (\cS_\cH, \cM_\cH, \cB(H_\cH))$ we have 
$$\chi_t(\mathcal{G}[\mathcal{H}]) \leq \chi_{b,t}(\mathcal{G})$$ where $b = \chi_t(\mathcal{H})$ and $t\in \{loc, q, qa, qc\}$.
\end{thm}

\begin{proof}
Let $\mathcal{L} = \mathcal{G}[\cH] = (\cS_\cL, \cM_\cL, \mathcal{B}(H_\cL))$. Then let
$\{P_a^\cH\}_{a=1}^b \subset \cM_\cH\bar\otimes \mathcal{N}_\cH$ be a minimal $t$-coloring of $\cH$, where $\mathcal{N}_\cH$ is a finite von Neumann algebra with faithful normal trace. Additionally, let $c = \chi_{b,t}(\mathcal{G})$ and let $\{P_a^\cG\}_{a=1}^c\subset \cM_\cG \bar\otimes \mathcal{N}_\cG $ be a set of commuting projections, where $\mathcal{N}_\cG $ is a finite von Neumann algebra with faithful trace, so that we have a minimal $t$-$b$-fold coloring. Recall that we may also desrcibe this $t$-$b$-fold coloring using the projections $Q_T = P_{t_1}^\cG P_{t_2}^\cG...P_{t_b}^\cG$ for $T = \{t_1,...,t_b\} \in \binom{[c]}{b}$.

Now, for $T \in \binom{[c]}{b}$ let $O_T = (t_1,...,t_b)$ be the ordered $b$-tuple where $t_i\in T$ and $t_i< t_j$ for $i,j\in \{1,2,...,b\}$ and $i< j$. Then define $\psi(T,a) = \ell$ where $\ell$ is such that $t_\ell = a$ (note that we only define $\psi$ when $a\in T$ so this is a well defined map). Then for all $a\in \{1,2,...,c\}$ we can define  

$$P_{a}^{\cL} = \sum_{T \in \binom{[c]}{b}} Q_T\odot P_{\psi(T,a)}^\cH$$
where $\odot$ is the permuted tensor product (see Notation \ref{notaiton:Permulted}). Denote $\cN = \cN_\cG\bar\otimes \cN_\cH$ so that we have $P_a^\cL \in \cM_\cG\bar\otimes \cM_\cH \bar\otimes \cN$. We shall show that
$\{P_a^{\cL}\}_{a=1}^c$ is a coloring of $\mathcal{L}$ using $c = \chi_{b,t}(\mathcal{G})$ colors.

First, for any $a\in \{1,2,...,c\}$ we get that $P_{a}^\cL$ is a projection. Indeed,
\begin{align*}
    (P_a^\cL)^* =& \left(\sum_{T\in\binom{[c]}{b}} Q_T\odot P_{\psi(T,a)}^\cH\right)^*
    = \sum_{T\in\binom{[c]}{b}} \left(Q_T^*\odot (P_{\psi(T,a)}^\cH)^*\right)
    = \sum_{T\in\binom{[c]}{b}} \left(Q_T\odot P_{\psi(T,a)}^\cH\right) = P_a^\cL.
\end{align*}
 Similarly, we get 
\begin{align*}
    (P_a^\cL)^2 &= \left(\sum_{T\in\binom{[c]}{b}} Q_T\odot P_{\psi(T,a)}^\cH\right)^2\\
    &= \sum_{T\in\binom{[c]}{b}}\left(\sum_{S\in\binom{[c]}{b}} \left(Q_T\odot P_{\psi(T,a)}^\cH\right)\left(Q_S\odot P_{\psi(S,a)}^\cH\right)\right)\\
    &=\sum_{T\in\binom{[c]}{b}} \left(Q_T\odot P_{\psi(T,a)}^\cH\right)^2\\
    &= \sum_{T\in\binom{[c]}{b}} \left(Q_T\odot P_{\psi(T,a)}^\cH\right) = P_a^\cL
\end{align*} where we used 
$$\left(Q_T\odot P_{\psi(T,a)}^\cH\right)\left(Q_S\odot P_{\psi(S,a)}^\cH\right) = (Q_TQ_S)\odot (P_{\psi(T,a)}^\cH P_{\psi(S,a)}^\cH) = 0$$
when $S\neq T$.

We will now show that $\sum P_a^\cL = I_{\cM_\cL\bar\otimes \mathcal{N}}$.  We have
\begin{align*}
    \sum_{a=1}^c P_a^\cL &=\sum_{a=1}^c \left(\sum_{T \in \binom{[c]}{b}} Q_T\odot P_{\psi(T,a)}^\cH\right)\\
    &= \sum_{T \in \binom{[c]}{b}}\left( Q_T\odot P_1^\cH + Q_T \odot P_2^\cH + ... + Q_T \odot P_b^\cH \right)\\
    &= \sum_T \left(Q_T\odot I_{\cM_\cH\bar\otimes \mathcal{N}_\cH}\right) = I_{\cM_\cL\bar\otimes\mathcal{N}}.
\end{align*}

Finally, we will show that for any $X\in \cS_L$ we have $P_a^\cL(X\otimes I_{\mathcal{N}}) P_a^\cL = 0$ for any $a\in \{1,2,3,..., c\}$.
We find 
\begin{align*}
    P_a^\cL(\cS_\cL\otimes I_{\mathcal{N}}) P_a^\cL &= \left(\sum_{T \in \binom{[c]}{b}} Q_S\odot P_{\psi(T,a)}^\cH\right) (\cS_\cL \otimes I_{\mathcal{N}})\left(\sum_{T \in \binom{[c]}{b}} Q_S\odot P_{\psi(T,a)}^\cH\right).
\end{align*}
So, using $\cS_\cL = \cS_\cG  \otimes \cB(H_\cH) + \cM_\cG '\otimes \cS_\cH$  we get 
\begin{align*}
    P_a^\cL(\cS_\cL\otimes I_{\mathcal{N}}) P_a^\cL &= P_a^\cL\left((\cS_\cG \otimes \cB(H_\cH) +\cM_\cG '\otimes \cS_\cH)\otimes I_{\cM_\cL\bar\otimes\mathcal{N}}\right)P_a^\cL\\
    &= \sum_{\substack{(S,T),\\S\neq T\in \binom{[c]}{b}}} \left((Q_S\odot P_{\psi(S,a)}^\cH)((\cS_\cG \otimes \cB(H_\cH)\otimes + \cM_\cG '\otimes \cS_\cH)\otimes I_{\cM_\cL\bar\otimes\mathcal{N}})(Q_T\odot P_{\psi(T,a)}^\cH)\right).
\end{align*}
In the calculation above, we can safely ignore the case where $S = T$. To see this, if $S=T$ then $Q_S  = Q_T$ and $P_{\psi(S,a)}^H = P_{\psi(T,a)}^H$, whence $Q_S (\cS_\cG  \otimes I_{\cM_\cG \bar\otimes\mathcal{N}_\cG })Q_S = 0$ and $P_a^H (\cS_\cH\otimes I_{\cM_\cH\bar\otimes\mathcal{N}_\cH}) P_a^H = 0$ for any $a\in \{1,2,...,b\}$ $S= \{s_1,...,s_b\} \subset \{1,2,...,c\}$. This implies that the corresponding summand is zero.  
Additionally, since $a\in S$ and $a\in T$ we have $Q_S (\cS_\cG \otimes I_{\cM_\cG \bar\otimes\mathcal{N}_\cG })Q_T  = 0$, so the first term vanishes when $S\neq T$. Also, since $Q_S\in \cM_\cG \otimes \mathcal{N}_\cG $, we have $Q_S$ commutes with $\cM'\otimes I_{\cM_\cG \bar\otimes\mathcal{N}_\cG }$. Thus, since $Q_SQ_T = 0$ if $S\neq T$, we also get that when $S\neq T$ the second term vanishes. 
Hence, we have $P_a^L(\cS_L\otimes I_\mathcal{N})P_a^L = 0,$ implying we have found a quantum coloring of $\mathcal{L} = \mathcal{G}[\cH]$ using $c = \chi_{b,t}(\mathcal{G})$ colors where $b = \chi_t(\mathcal{H})$. Therefore,  we have $\chi_t(\mathcal{L}) \leq \chi_{b,t}(\mathcal{G})$.

\end{proof}

Combining Theorem \ref{thm:lexicographicBoundedByProduct} and Corollary \ref{item:bFoldVersusChromatic}
, we get the following corollary.
\begin{cor}
For quantum graphs $\mathcal{G}$ and $\mathcal{H}$ we have 
$$\chi_t(\mathcal{G}[\mathcal{H}]) \leq \chi_{b,t}(\mathcal{G})\leq \chi_t(\cG)\chi_t(\cH)$$ where $b = \chi_t(\mathcal{H})$ and $t\in \{loc, q, qa, qc\}$.
\label{cor:productLexBound}
\end{cor}

\subsection{Bounds for Strong Products}
In this section we establish bounds and consider edge cases for chromatic numbers for the strong product of graphs. This product was not included in the work of \cite{graph_products}, and to our knowledge no work has been done for non-local games considering such graph products for either classical or quantum graphs. It is known \cite{HandbookGraphProducts} that for classical graphs $G_1$ and $G_2$ we have 
$\max\{\chi(G_1),\chi(G_2)\} \leq \chi(G_1\boxtimes G_2) \leq \chi(G_1)\chi(G_2)$. Here we will extend both such bounds to strong products of quantum graphs.

The lower bound is obtained as a corollary of Theorem \ref{thm:BoundCartesianProduct}.

\begin{cor}
For quantum graphs $\cG$ and $\cH$,
    $$\max\{\chi_t(\cG), \chi_t(\cH)\} \leq \chi_t(\cG\boxtimes\cH)$$
    for all $t\in \{loc, q, qa, qc\}$
    \label{cor:strongLowerBound}
\end{cor}
\begin{proof}
    Using Theorem \ref{thm:BoundCartesianProduct}  and Corollary \ref{cor:q_color_transitive} we need only show that $\cG\square \cH \to_{loc} \cG\boxtimes \cH$. However, considering the Definitions \ref{cartesion_prod_def} and \ref{def:strongProduct} it is clear that the identity map defines such a local quantum graph homomorphism (suppressing $\otimes\C$ as it may be neglected).
\end{proof}

The upper bound requires slightly more work.

\begin{thm}
For quantum graphs $\mathcal{G} = (\cS_\cG ,\cM_\cG , \cB(H_\cG ))$  and $\mathcal{H} = (\cS_\cH, \cM_\cH, \cB(H_\cH))$, we have $$\chi_t(\mathcal{G}\boxtimes\mathcal{H}) \leq \chi_t(\mathcal{G})\chi_t(\mathcal{H})$$
$\forall t\in \{loc, q, qa, qc\}.$
\label{thm:strong_upper_bound}
\end{thm}

\begin{proof}
Let $\{P_a^\cG\}_{a=1}^{\chi_t(\cG)}\subseteq \mathcal{M}_\cG \bar\otimes \mathcal{N}_\cG $ and $\{P_b^\cH\}_{b=1}^{\chi_t(\cH)}\subseteq \mathcal{M}_\cH\bar\otimes \mathcal{N}_\cH$ be minimal $t$-colorings of $\cG$ and $\cH$. 
Further, let $\cG\boxtimes \cH = \left(\cS_{\cG\boxtimes \cH}, \cM_{\cG\boxtimes \cH}, \cB(H_{\cG\boxtimes \cH})\right)$.

Define $\{R_{a,b}\}_{a,b=1}^{a=\chi_t(\cG),b=\chi_t(\cH)}$ via $R_{a,b} = P_a^\cG \odot P_b^\cH$ where $\odot$ is the permuted tensor product (see Notation \ref{notaiton:Permulted}), and denote $\cN = \cN_\cG \bar\otimes \cN_\cH$. Note that tensoring von Neuman algebras which are 1-dimensional, finite dimensional, $\mathcal{R}^\omega$ embeddable, or finite results in an algebra with the same quality.

Then  we have 
$$R_{a,b}^2 = \left(P_a^\cG\odot P_b^\cH\right)\left(P_a^\cG \odot P_b^\cH\right) = \left(P_a^\cG\right)^2 \odot \left(P_b^\cH\right)^2 = P_a^\cG \odot P_b^\cH = R_{a,b}$$
and 
$$R_{a,b}^* = \left(P_a^\cG\odot P_b^\cH\right)^* = \left(P_a^\cG\right)^* \odot \left(P_b^\cH\right)^* = P_a^\cG\odot P_b^\cH = R_{a,b},$$
so $R_{a,b}$ is a projection.

Additionally, we have $$\sum_{a,b}R_{a,b} = I_{\cM_\cG\bar\otimes \cN_\cG} \odot I_{\cM_\cH\bar\otimes \cN_\cH} = I_{(\mathcal{M}_\cG \bar\otimes \mathcal{M}_\cH)\bar\otimes \mathcal{N}}.$$

Finally, noting that the operator space for $\cG\boxtimes\cH$ is $\cS_{G\boxtimes H} = \cS_\cG \otimes \cM'_\cH + \cM_\cG '\otimes \cS_\cH + \cS_\cG \otimes\cS_\cH$, we get that $\forall X\in \cS_{G\boxtimes H}$ we have $X = X_\cG \otimes D_\cH + D_\cG \otimes X_\cH + Y_\cG \otimes Y_\cH$ for some $X_\cG , Y_\cG \in \cS_\cG ,\, X_\cH, Y_\cH\in \cS_\cH,\, D_\cG \in \cM_\cG '$ and $D_\cH \in \cM_\cH'$. Thus, 
\begin{align*}R_{a,b}(X\otimes I_{\mathcal{N}})R_{a,b} = R_{a,b}((X_\cG \otimes D_\cH)\otimes I_{\mathcal{N}})R_{a,b} + R_{a,b}\left((D_\cG \otimes X_\cH)\otimes I_{\mathcal{N}}\right)R_{a,b} + R_{a,b}((Y_\cG \otimes Y_\cH)\otimes I_\cN)R_{a,b} = 0,\end{align*}
because $P_a^\cG(X_\cG \otimes I_{\mathcal{N}_\cG })P_a^\cG = 0$, $P_a^\cG(Y_\cG \otimes I_{\cN_\cG })P_a^\cG = 0$, and 
$P_b^\cH(X_\cH\otimes I_{\mathcal{N}_\cH})P_b^\cH = 0.$

Together, this gives us that $\{R_{a,b}\}$ is a $t$ coloring of $\cG\boxtimes\cH$ using $\chi_t(\cG)\chi_t(\cH)$ colors, so we must have $\chi_t(\cG\boxtimes\cH) \leq \chi_t(\cG)\chi_t(\cH).$
\end{proof}

\begin{remark}
We easily see from \cite[Proposition 6.4]{GanesanHarrisQuantumToClassical} that Theorem \ref{thm:strong_upper_bound} gives us $\chi_t(\cG)\chi_t(\cH)$ is an upper bound for the Cartesian product (and categorical product, but we already have a stronger bound in this case). 
\label{remark:productUpperBound}
\end{remark}

We now comment upon the cases where these bounds for the strong product are sharp.

\begin{prop}
    For quantum graphs $\cG = (\cS_\cG , \cM_\cG , \cB(H_\cG ))$ and $\cH = (\cS_\cH, \cM_\cH, \cB(H_\cH))$, the upper bound found in Theorem \ref{thm:strong_upper_bound} is tight when $\cG$ and $\cH$ are both complete graphs. Furthermore, the lower bound in Corollary \ref{cor:strongLowerBound} is tight when $\cG$ (resp. $\cH$) have $\cS_\cG $ (resp. $\cS_\cH$) equal to zero.
\end{prop}
\begin{proof}
    Let $\cG\boxtimes \cH = \left(\cS_{\cG\boxtimes \cH}, \cM_{\cG\boxtimes \cH}, \cB(H_{\cG\boxtimes \cH})\right)$.
    
    If $\cG$ and $\cH$ are complete quantum graphs then $\cS_\cG  = M_n$ and $\cS_\cH = M_m$ where $n= \dim H_\cG $ and $m = \dim H_\cH$. Then, unless both $\cM_\cG $ and $\cM_\cH$ are abelian, both at least one of the initial and the final product local chromatic numbers will be infinite. If both quantum graphs are abelian then we are in the fully classical setting when $t=loc$ so the theorem is already known.
    We now assume that $t\neq loc$.  Then, using Definition \ref{def:strongProduct}, we see that $\cS_{\cG\boxtimes \cH} = M_{nm}$ so we have $\cG\boxtimes \cH$ is a complete quantum graph as well. Thus, we have $\chi_t(\cG) = n, \chi_t(\cH) = m$ and $\chi_t(\cG\boxtimes\cH) = mn$ via \cite[Theorem 6.9]{GanesanHarrisQuantumToClassical}. Hence, the first part of the Proposition is shown.

    To show the second half, we may assume without loss of generality that $\cS_\cH = 0$ and not put any restriction on $\cS_\cG $. We wish to show that $\chi_t(\cG\boxtimes\cH) = \chi_t(\cG)$. We accomplish this by noting that $\cG\to_{loc} \cG\boxtimes\cH$ and $\cG\boxtimes\cH\to_{loc}\cG$ with the first quantum graph homomorphism accomplished using the CPTP map that is the composition of the maps in Theorem \ref{thm:BoundCartesianProduct} and Corollary \ref{cor:strongLowerBound}, and the second quantum graph homomorphism accomplished via the partial trace defined in Theorem \ref{thm:BoundCategoricalProduct} that takes the operator spaces to the correct place because $\cS_\cH = 0$.
\end{proof}

Finally, corollary \ref{cor:productLexBound} along with Theorem \ref{thm:strong_upper_bound} and Remark \ref{remark:productUpperBound} gives the following result.
\begin{cor} 
Let $t\in \{loc, q, qa, qc\}$.
    For quantum graphs $\mathcal{G} = (\cS_\cG , \cM_\cG , \cB(H_\cG ))$ and $\mathcal{H} = (\cS_\cH, \cM_\cH, \cB(H_\cH))$ we have 
$$\chi_t(\mathcal{G}[\mathcal{H}]),\chi_t(\mathcal{G}\Box \mathcal{H}), \chi_t(\mathcal{G}\times \mathcal{H}),\chi_t(\mathcal{G}\boxtimes \mathcal{H}) \leq \chi_t(\cG)\chi_t(\cH).$$ 
\end{cor}




\bibliographystyle{amsalpha}
\bibliography{references.bib}


\end{document}